\documentclass[11pt,a4paper,twoside]{article}

\usepackage{amssymb,amscd,amsfonts,amsbsy,amsmath,amsthm}
\usepackage{enumerate}
\usepackage{mathrsfs}
\usepackage{epsf,epsfig,esint}
\usepackage{pdfsync}
\usepackage[all]{xy}
\usepackage{pdfpages}

\usepackage[colorlinks=true,citecolor=magenta,linkcolor=cyan,backref]{hyperref}
\AtBeginDocument{}

\DeclareMathOperator*{\esssup}{ess\,sup}

\numberwithin{equation}{section}

\newtheorem{proposition}{Proposition}[section]
\newtheorem{theorem}[proposition]{Theorem}
\newtheorem{lemma}[proposition]{Lemma}
\newtheorem{corollary}[proposition]{Corollary}

\theoremstyle{definition}
\newtheorem{definition}[proposition]{Definition}

\newtheorem{remark}[proposition]{Remark}

\newtheorem{notation}[proposition]{Notation}
\newtheorem{example}[proposition]{Example}

\title{Hodge-Dirac, Hodge-Laplacian and Hodge-Stokes operators\\
in $L^p$ spaces on Lipschitz domains\,
}

\author{Alan McIntosh\,
\thanks{Mathematical Sciences Institute,
Australian National University, Canberra, ACT 2601, Australia
- {\tt email: alan.mcintosh@anu.edu.au}}
\and Sylvie Monniaux\,
\thanks{Aix-Marseille Universit\'e, CNRS, Centrale Marseille, I2M,  
Marseille, France - {\tt email: sylvie.monniaux@univ-amu.fr}}
}

\date{}

\begin{document}

\maketitle

\begin{abstract}  
This paper concerns Hodge-Dirac operators $D_{{}^\Vert}=d+\underline{\delta}$ 
acting in $L^p(\Omega, \Lambda)$ where $\Omega$ is a bounded open subset 
of ${\mathbb{R}}^n$ satisfying some kind of Lipschitz condition, $\Lambda$ is 
the exterior algebra of ${\mathbb{R}}^n$, $d$ is the exterior derivative acting on 
the de Rham complex of differential forms on $\Omega$, and $\underline{\delta}$ 
is the interior derivative with tangential boundary conditions. In 
$L^2(\Omega,\Lambda)$, $\underline{\delta} = {d}^*$ and $D_{{}^\Vert}$ is 
self-adjoint, thus having bounded resolvents 
$\bigl\{({\rm I}+itD_{{}^\Vert})^{-1}\bigr\}_{t\in{\mathbb{R}}}$ as well as a bounded 
functional calculus in $L^2(\Omega,\Lambda)$. We investigate the range of 
values $p_H<p<p^H$ about $p=2$ for which $D_{{}^\Vert}$ has bounded 
resolvents and a bounded holomorphic functional calculus in 
$L^p(\Omega,\Lambda)$. On domains which we call very weakly Lipschitz, 
we show that this is the same range of values as for which 
$L^p(\Omega,\Lambda)$ has a Hodge (or Helmholz) decomposition, being an 
open interval that includes 2.

The Hodge-Laplacian $\Delta_{{{}^\Vert}}$ is the square of the Hodge-Dirac 
operator,  i.e. $-\Delta_{{}^\Vert}={D_{{}^\Vert}}^2$, so it also has a bounded 
functional calculus in $L^p(\Omega,\Lambda)$ when $p_H<p<p^H$. 
But the Stokes operator with Hodge boundary conditions, which is the restriction 
of $-\Delta_{{}^\Vert}$ to the subspace of divergence free vector fields in 
$L^p(\Omega,\Lambda^1)$ with tangential boundary conditions, has a bounded 
holomorphic functional calculus for further values of $p$, namely for 
$\max\{1,{p_H}_S\}<p<p^H$ where ${p_H}_S$ is the Sobolev exponent 
below $p_H$, given by $1/{{p_H}_S} =1/{p_H}+1/n$, so that 
${{p_H}_S}<2n/(n+2)$. In 3 dimensions, ${p_H}_S<6/5$. 

We show also that for bounded strongly Lipschitz domains $\Omega$, 
$p_H<2n/(n+1)<2n/(n-1)<p^H$, in agreement with the known results that 
$p_H < 4/3<4<p^H$ in dimension 2, and $p_H<3/2< 3<p^H$ in dimension 3. 
In both dimensions 2 and 3, ${p_H}_S<1$, implying that the Stokes operator 
has a bounded functional calculus in $L^p(\Omega,\Lambda^1)$ when 
$\Omega$ is strongly Lipschitz and $1<p<p^H$.
\end{abstract}


\section{Introduction}
\label{sec:intro}

In this paper, we take a first order approach to developing an $L^p$ theory for the 
Hodge-Laplacian and the Stokes operator with Hodge boundary conditions, acting 
on a bounded open subset $\Omega$ of ${\mathbb{R}}^n$. In particular, we give 
conditions on $\Omega$ and $p$ under which these operators have bounded 
resolvents, generate analytic semigroups, have bounded Riesz transforms, or 
have bounded holomorphic functional calculi. The first order approach of initially 
investigating the Hodge-Dirac operator, provides a framework for strengthening 
known results and obtaining new ones on general classes of domains, in what we 
believe is a straightforward manner. 

In particular we consider the usual strongly Lipschitz and weakly Lipschitz domains 
(see Section \ref{subsec:weakLip}), but mostly we only need the still weaker 
concept of a {\tt very weakly Lipschitz domain} $\Omega$, by which we mean 
that $\Omega=\bigcup_{j=1}^M\Omega_j$ where each $\Omega_j$ is a 
bilipschitz transformation of the unit ball, and 
$1\!{\rm l}_\Omega=\sum_{j=1}^M \chi_j$ for some Lipschitz functions 
$\chi_j:\Omega\to[0,1]$ with ${\rm sppt}_\Omega\chi_j\subset \Omega_j$.  

When $1<p<\infty$, we consider the {\tt exterior derivative} 
$d=\nabla\wedge$ as an unbounded operator in the space 
$L^p(\Omega,\Lambda)$ with domain 
${\rm{\sf D}}^p(d)= \{u\in L^p(\Omega,\Lambda)\,;\,d u \in L^p(\Omega,\Lambda)\}$, 
where $\Lambda=\Lambda^0\oplus\Lambda^1\oplus\dots\oplus\Lambda^n$ is the 
exterior algebra of ${\mathbb{R}}^n$ and 
$L^p(\Omega,\Lambda) = \oplus_{k=0}^n L^p(\Omega,\Lambda^k)$ is the space 
of differential forms on $\Omega$.  
We shall see that on a very weakly Lipschitz domain $\Omega$, the range 
${\rm{\sf R}}^p(d)$ of the exterior derivative is a closed subspace of the null space 
${\rm{\sf N}}^p(d)$ with finite codimension. Similar results hold for the 
{\tt interior derivative} $\delta =-\nabla\lrcorner\ $.

The duals of the operators $d$ and $\delta$ in $L^{p'}(\Omega,\Lambda)$ are 
denoted by $\underline{\delta}$ and $\underline{d}$, being restrictions of the 
operators $\delta$ and $d$ to smaller domains, namely to the completion 
of ${\mathscr{C}}_c^\infty(\Omega)$ in the graph norms. By duality, the range 
${\rm{\sf R}}^p(\underline{\delta})$ is a closed subspace of the null space 
${\rm{\sf N}}^p(\underline{\delta})$ with finite codimension, and similarly for 
$\underline{d}$.
We remark that when $\Omega$ is weakly Lipschitz, so that the unit normal 
$\nu$ is defined a.e.~on the boundary $\partial\Omega$, then $\underline{\delta}$
and $\underline{d}$ have domains ${\rm{\sf D}}^p(\underline{\delta}) 
= \{u\in{\rm{\sf D}}^p(\delta)\,;\, \nu\lrcorner\, u_{|_{\partial\Omega}} = 0\}$ and 
${\rm{\sf D}}^p(\underline{d}) = \{u\in{\rm{\sf D}}^p(d)\,;\,
\nu\wedge u_{|_{\partial\Omega}}= 0\}$ (called {\tt tangential} and {\tt normal 
boundary conditions} respectively). 

When $p=2$ and $\Omega$ is very weakly Lipschitz, then $\underline{\delta}= d^*$,
so the {\tt Hodge-Dirac operator} $D_{{}^\Vert}= d +\underline{\delta}$ is 
self-adjoint in $L^2(\Omega,\Lambda)$, and thus has bounded resolvents 
$\{({\rm I}+itD_{{}^\Vert})^{-1}\}_{t\in{\mathbb{R}}}$ as well as a bounded 
functional calculus in $L^2(\Omega,\Lambda)$.
Moreover there is a {\tt Hodge decomposition}
$$
L^2(\Omega,\Lambda) = {{\rm{\sf R}}^2(d)} \overset\perp\oplus 
{{\rm{\sf R}}^2(\underline{\delta})} \overset\perp\oplus {\rm{\sf N}}^2(D_{{}^\Vert})
$$
where the space of harmonic forms ${\rm{\sf N}}^2(D_{{}^\Vert}) 
= {\rm{\sf N}}^2(d)\cap{\rm{\sf N}}^2(\underline{\delta})$ is finite-dimensional 
(owing to the finite codimension of ${\rm{\sf R}}^2(\underline{\delta})$ in 
${\rm{\sf N}}^2(\underline{\delta})$). Similar results hold for $D_{{}^\bot} 
= \underline{d}+\delta$.  

When $\Omega$ is smooth (see, e.g., \cite{Schw95}), then each of 
these $L^2$ results has an $L^p$ analogue for all $p \in (1,\infty)$ (provided we 
drop orthogonality from the definition of the Hodge decomposition). This is known 
not to be the case on all Lipschitz domains, though typically $L^p$ results do hold 
for all $p$ sufficiently close to 2. In this paper we prove that the following results 
hold, provided that $\Omega$ is a very weakly Lipschitz domain.
\begin {itemize}
\item
There exist Hodge exponents $p_H$, $p^H = {p_H}'$  with 
$1\le p_H<2<p^H\le\infty$ such that the Hodge decomposition
\begin{equation*}
L^p(\Omega,\Lambda)={\rm{\sf R}}^p(d)\oplus
{\rm{\sf R}}^p(\underline{\delta})\oplus
({\rm{\sf N}}^p(d) \cap{\rm{\sf N}}^p(\underline{\delta}))
\end{equation*}
holds if and only if $p_H<p<p^H$.
Moreover, for $p$ in this range, $D_{{}^\Vert}=d+\underline{\delta}$ is a closed 
operator in $L^p(\Omega,\Lambda)$, and 
${\rm{\sf N}}^p(d) \cap{\rm{\sf N}}^p(\underline{\delta}) = {\rm{\sf N}}^p(D_{{}^\Vert})
= {\rm{\sf N}}^2(D_{{}^\Vert})$. (Theorem~\ref{thm:Hodge-dec-Lp})
\item 
The Hodge-Dirac operator $D_{{}^\Vert}$ is bisectorial with a bounded 
holomorphic functional calculus in $L^p(\Omega,\Lambda)$ if and only if
$p_H<p<p^H$; in particular, for each such $p$ there exists $C_p >0$ such that 
$\|({\rm I}+itD_{{}^\Vert})^{-1}u\|_p \le  C_p\|u\|_p$ for all $t\in {\mathbb{R}}$ 
(Theorem~\ref{thm:mainThm4.1} \eqref{5.1(i)} and \eqref{5.1(ii)}).
\item 
When $p_H<p<p^H$, the {\tt Hodge-Laplacian} 
$\Delta_{{}^\Vert}=-{D_{{}^\Vert}}^2=-( d\underline{\delta}+\underline{\delta}d)$ is 
sectorial with a bounded holomorphic functional calculus in 
$L^p(\Omega,\Lambda)$ and has a bounded Riesz transform in the sense that 
$\|\sqrt{-\Delta_{{}^\Vert}}\, u\|_p \approx \|D_{{}^\Vert} u\|_p$; in particular, 
$\|({\rm I}+t^2\Delta_{{}^\Vert})^{-1}u\|_p \le  {C_p}^2\|u\|_p$ for all $t>0$, and 
$\Delta_{{}^\Vert}$ generates an analytic semigroup in $L^p(\Omega,\Lambda)$
(Corollary~\ref{cor:HodgeLaplacian}). Let us mention that sectoriality 
(\cite[Theorems~6.1 and~7.1]{MM09a}) and boundedness of Riesz transforms
(\cite[Theorem~5.1]{HMM11}) have already been proved in in the case of 
bounded strongly Lipschitz domains.
\item 
If $\max\{1,{p_H}_S\}<p<p^H$ (where ${p_H}_S = np_H/(n +p_H) <2n/(n+2)$) 
then the operators $f(D_{{}^\Vert})$ in the holomorphic functional calculus of 
$D_{{}^\Vert}$, are bounded on ${\rm{\sf N}}^p(\underline{\delta})$ and on 
${\rm{\sf N}}^p(d)$ (Theorem~\ref{thm:mainThm4.1} \eqref{5.1(iii)}).
\item 
When $\max\{1,{p_H}_S\}<p<p^H$, the restriction of the  Hodge-Laplacian 
$\Delta_{{}^\Vert}$ to ${\rm{\sf N}}^p(\underline{\delta})$ is sectorial with a bounded 
holomorphic functional calculus; in particular, the estimate
$\|({\rm I}-t^2\Delta_{{}^\Vert})^{-1}u\|_p \le {C_p}^2\|u\|_p$ holds for all 
$u \in {\rm{\sf N}}^p(\underline{\delta})$ and all $t>0$, and $\Delta_{{}^\Vert}$ 
generates an analytic semigroup on ${\rm{\sf N}}^p(\underline{\delta})$. 
The corresponding results also hold on ${\rm{\sf N}}^p(d)$
(Corollary~\ref{cor:HodgeStokes}).
\item 
If $\Omega$ is strongly Lipschitz, then $p_H<2n/(n+1)<2n/(n-1)<p^H$ and 
${p_H}_S< 2n/(n+3)$, in particular $\max\{1,{p_H}_S\} =1$ in dimensions 2 and 3. 
(Theorem \ref{thm:pHlip})
\end{itemize}
The last two points are of particular relevance to the Stokes operator with 
Hodge boundary conditions, which is the restriction of $-\Delta_{{}^\Vert}$ to 
$\{u\in L^p(\Omega,\Lambda^1)\,;\, \underline{\delta}u = 0\}$. In dimension
$n=3$, the last point shows that the Stokes operator has a bounded holomorphic 
functional calculus for all $p\in(1,p^H)$ where $p^H>3$ depends
on $\Omega$. This result completes the result stated in
\cite[Theorem~7.2]{MM09a}, where only sectoriality for $p\in(p_H,p^H)$
has been proved.

A similar lower Hodge exponent arises when considering perturbed 
Hodge-Dirac operators of the form $D_{{}^\Vert,B}=d+\underline{\delta}_B
=d+B^{-1}\underline{\delta}B$, where $B,B^{-1} \in 
L^\infty(\Omega,{\mathscr{L}}(\Lambda))$ with $\text{Re} B \ge\kappa {\rm I}$, 
which we shall only do in the case of bounded strongly Lipschitz domains.  
In this case, all of the above points, except for the final one, hold with 
$\underline{\delta}$ replaced by $\underline{\delta}_B$, $D_{{}^\Vert}$ replaced 
by $D_{{}^\Vert,B}$, and $\Delta_{{}^\Vert}$ replaced by 
$\Delta_{{}^\Vert,B}= -(D_{{}^\Vert,B})^2$, though of course the Hodge exponents
depend on $B$, with  $p^H$ possibly unequal to ${p_H}'$.   
See Section~\ref{sec:pertDirac}.

Our proofs of the results announced above rely strongly on the potential 
maps defined in Section~\ref{sec:potentials}. Those maps can be of independent
interest. They are refined versions of the ones developed in \cite{MMM08} and 
\cite{CMcI10}, refined in two ways:
\begin{itemize}
\item
we can deal here with very weakly Lipschitz domains while \cite{MMM08}
and \cite{CMcI10} only treat the case of bounded strongly Lipschitz domains;
\item
we obtain true potentials, in the sense that the maps $R$, $S$, $T$ and
$Q$ defined in Section~\ref{sec:potentials} are right inverses of $d$, 
$\underline{\delta}$, $\underline{d}$ and $\delta$ on their ranges.
\end{itemize}
As a direct consequence, the families of ranges and of nulspaces of these 
operators in $L^p$, $1<p<\infty$, form complex interpolation scales
(see Corollary~\ref{cor:complex-interpolation-scale}).

\medskip

In the case of ${\mathbb{R}}^n$, results in the same spirit (extending the range of
$p$ for which a bounded holomorphic functional calculus holds outside the Hodge
range) have been recently obtained in \cite{FMcIP14} and \cite{AS14}. 
The methods used there are different, and specific to ${\mathbb{R}}^n$.

\subsection{Acknowledgements}
The authors appreciate the support of the Mathematical Sciences Institute 
at the Australian National University, Canberra (Australia), where much of the 
collaboration took place, as well as of the Institut de Math\'ematiques de Marseille, 
Universit\'e Aix-Marseille (France). Both authors were supported by the Australian 
Research Council. The second author also aknowledges the partial support by the ANR project 
``Harmonic analysis at its boundaries" ANR-12-BS01-0013.
Our understanding of this topic has benefited from discussions with Pascal Auscher,
Dorothee Frey, Pierre Portal, Andreas Ros\'en, as well as of previous works with 
Dorina Mitrea and Marius Mitrea.

\section{Setting}
\label{sec:setting}

In this section, we specify some concepts used throughout the paper. At all 
times we are considering functions and operators defined on bounded open 
subsets $\Omega$ of Euclidean space ${\mathbb{R}}^n$ with dimension $n\ge 2$.

\subsection{Notation}

\begin{notation} 
\label{not:p,p',pS,p*}
For $1\le p\le \infty$, we denote by $p'$ the H\"older conjugate exponent, i.e.,
$\frac{1}{p}+\frac{1}{p'}=1$ (with the convention that $\frac{1}{\infty}=0$), 

\noindent
by $p_S$ the lower Sobolev exponent defined by 
$\frac{1}{p_S}=\frac{1}{p}+\frac{1}{n}$,

\noindent
and by $p^*$  the exponent for which
$W^{\frac{1}{p},p}({\mathbb{R}}^n)\hookrightarrow L^{p^*}({\mathbb{R}}^n)$, 
i.e., $p^*=\frac{np}{n-1}$. 

\noindent
We denote by $p^S$ the Sobolev exponent given by
$\frac{1}{p^S}=\frac{1}{p}-\frac{1}{n}$ if $1\le p<n$, $p^S=\infty$ if $p>n$.
If $p=n$, $p^S$ is multivalued, it takes any value in $[p,\infty)$.
\end{notation}

\begin{remark}
\label{rem:qS'}
Note that if $p\in[1,n)$, then $(p^S)'\in(1,n]$ and $(p^S)'=(p')_S$. Note also that if 
$r\in(1,\infty)$, then $(r^*)'\in(1,n)$ and 
\begin{equation}
\label{eq:r*'S}
\bigl((r^*)'\bigr)^S=(r')^*.
\end{equation}
\end{remark}

\begin{notation}
\label{not:sectors}
The following sectors in the complex plane will be considered:
\begin{align*}
&S_{\mu+}^\circ:=\bigl\{z\in{\mathbb{C}}\setminus\{0\};|\arg z|<\mu\bigr\}
\quad\mbox{and}\quad S_{\mu+}:=\overline{S_{\mu+}^\circ}\quad
\mbox{if }\mu\in (0,\pi), 
\\
&S_{\mu-}^\circ:=-S_{\mu+}^\circ \quad \mbox{and}\quad
S_\mu^\circ:=S_{\mu+}^\circ\cup S_{\mu-}^\circ\quad\mbox{if }
\mu\in \bigl(0,\tfrac{\pi}{2}\bigr)
\\
&S_\mu:=\overline{S_\mu^\circ}\quad\mbox{if }
\mu\in \bigl(0,\tfrac{\pi}{2}\bigr) \quad\mbox{and}\quad
S_0:={\mathbb{R}}\times\{0\}\subset{\mathbb{C}}. 
\end{align*}
\end{notation}

\begin{notation}
\label{not:D,N,R}
The domain of an (unbounded linear) operator $A$ is denoted by ${\rm{\sf D}}(A)$,
its null space by ${\rm{\sf N}}(A)$, its range by ${\rm{\sf R}}(A)$, and its graph 
by ${\mathsf{G}}(A)$. When the operator $A$ acts in $L^p(\Omega)$, these are 
sometimes written as ${\rm{\sf D}}^p(A,\Omega)$, ${\rm{\sf N}}^p(A,\Omega)$,  
${\rm{\sf R}}^p(A,\Omega)$, and ${\mathsf{G}}^p(A,\Omega)$.
\end{notation}

\begin{notation}
\label{not:dist}
For $E,F\subset{\mathbb{R}}^n$ two Borel sets, denote by 
${\rm dist}\,(E,F)$ the distance between $E$ and $F$ defined by
${\rm dist}\,(E,F)=\inf\bigl\{|x-y|\,;\, x\in E,\, y\in F\bigr\}$.

\noindent
For a distribution $f$ defined on an open subset $\Omega$ of ${\mathbb{R}}^n$,
we denote the support of $f$ by ${\rm sppt_\Omega}f$ or sometimes just by 
${\rm sppt}\,f$.
\end{notation}

\begin{notation}
\label{not:unit_ball}
We denote by $B(x,r)$ the ball in ${\mathbb{R}}^n$ with centre 
$x\in {\mathbb{R}}^n$ and radius $r>0$, and set $B_\Omega(x,r) = 
B(x,r)\cap\Omega$, namely the ball in $\Omega$ with the same centre and 
radius. 
\end{notation}

\subsection{Various types of Lipschitz domains}
\label{subsec:weakLip}

In the following definitions and properties, we follow the paper \cite{AMcI03} 
by Axelsson (now Ros\'en) and the first author. 
By a bounded weakly Lipschitz domain we mean a bounded open set $\Omega$
separated from the exterior domain ${\mathbb{R}}^n\setminus\overline{\Omega}$
by a weakly Lipschitz interface 
$\Sigma=\partial\Omega=\partial({\mathbb{R}}^n\setminus\overline{\Omega})$, 
defined as follows.

\begin{definition}
\label{def:uniflocLip}
Let  $\Omega \subset{\mathbb{R}}^n$ be an open set.
A function $f:\Omega\to{\mathbb{R}}^p$ is said to be {\tt uniformly
locally Lipschitz} (or {\tt Lipschitz} for short) if there exists $C>0$ 
such that for all $x\in\Omega$ there exists $r_x>0$ such that $|f(y)-f(z)|\le C|y-z|$ 
for all $y,z\in B_\Omega(x,r_x)$.
\end{definition}

We remark that every such function $f$ is differentiable a.e.~with derivatives 
$\partial_j f\in L^\infty(\Omega,{\mathbb{R}}^p)$.  

\begin{example}
\label{ex:LipNOTgloballyLip}
Let 
$$
\Omega:=\bigl\{(x,y)\in{\mathbb{R}}^2\,;\, 0<x^2+y^2<1, |\arg(x,y)|<\pi\bigr\} 
$$ 
and define $f:\Omega\to{\mathbb{R}}$ by $f(x,y)=(x^2+y^2)^{\frac12}\arg(x,y)$. 
Then $f$ is a
uniformly locally Lipschitz function in the sense of Definition~\ref{def:uniflocLip}, 
but not globally Lipschitz; i.e., there is no $C>0$ such that
$|f(z)-f(w)|\le C|z-w|$ for all $z,w\in\Omega$.
\end{example}

\begin{definition}
\label{def:bilip}
Let $\Omega\subset{\mathbb{R}}^n$ and let 
$\rho:\Omega\to \rho(\Omega)\subset{\mathbb{R}}^n$. We say that 
$\rho$ is a {\tt bilipschitz map} if $\rho$ is a bijective map from $\Omega$ to 
$\rho(\Omega)$ and $\rho$ and $\rho^{-1}$ are both uniformly locally Lipschitz.
\end{definition}

\begin{definition}
\label{def:weakLip}
The interface $\Sigma$ is {\tt weakly Lipschitz} if, for all $y\in \Sigma$, there is
a neighbourhood $V_y\ni y$ and a global bilipschitz map 
$\rho_y:{\mathbb{R}}^n\to{\mathbb{R}}^n$ such that
\begin{align*}
\Omega\cap V_y&=\rho_y\bigl({\mathbb{R}}^{n-1}\times(0,+\infty)\bigr)\cap V_y,
\\
\Sigma\cap V_y&=\rho_y\bigl({\mathbb{R}}^{n-1}\times\{0\}\bigr)\cap V_y,
\\
\bigl({\mathbb{R}}^n\setminus\overline{\Omega}\bigr)\cap V_y
&=\rho_y\bigl({\mathbb{R}}^{n-1}\times(-\infty,0)\bigr)\cap V_y.
\end{align*}
\end{definition}

\noindent
A special case of a weakly Lipschitz domain is a strongly Lipschitz domain
defined as follows.

\begin{definition}
\label{def:strongLip}
A {\tt strongly Lipschitz domain} is a weakly Lipschitz domain such that for 
all $y\in \Sigma$, there is a neighbourhood $V_y\ni y$ and a global bilipschitz 
map $\rho_y:{\mathbb{R}}^n\to{\mathbb{R}}^n$ satisfying the conditions of
Definition~\ref{def:weakLip} that takes the form
$$
\rho_y(x)=E_y\bigl(x',x_n-g_y(x')\bigr), 
\quad x=(x',x_n),\ x'=(x_1,\dots,x_{n-1})
$$ 
where $g_y:{\mathbb{R}}^{n-1}\to{\mathbb{R}}$ is a 
Lipschitz function such that $g_y(0)=0$ and $E_y$ is a Euclidian transformation.
\end{definition}

Reasoning as in \cite[Proof of Theorem~1.3]{AMcI03}, we see that bounded weakly 
Lipschitz domains have the following property. 

\begin{remark}
\label{rem:weakLip}
By Definition~\ref{def:weakLip} it follows that there exist bilipschitz maps
$\rho_j:B\to\rho_j(B)=:\Omega_j\subset\Omega$ ($j=1,\dots,M$) 
(where $B=B(0,1)$ denotes the unit ball in ${\mathbb{R}}^n$) such that 
$\displaystyle{\Omega=\bigcup_{j=1}^M\Omega_j}$, and there exist Lipschitz 
functions $\chi_j:\Omega\to[0,1]$ such that ${\rm sppt}_\Omega\chi_j\subset 
\Omega_j$ and $\sum_{j=1}^M \chi_j=1$ on $\Omega$.

Furthermore, we may assume that for each $j=1,\dots,M$, $\rho_j$ extends 
to a bilipschitz map between slightly larger open sets. 
\end{remark} 

\begin{example}
\label{ex:VWLdomain}
An important example of a weakly Lipschitz domain that is not
strongly Lipschitz is the ``two brick" domain in ${\mathbb{R}}^3$ 
defined as the interior of
\begin{align*}
&\bigl\{(x,y,z)\in{\mathbb{R}}^3;0\le z\le1, -2\le y\le2,-1\le x\le1\bigr\}\\
\cup&
\bigl\{(x,y,z)\in{\mathbb{R}}^3;-1\le z\le0, -1\le y\le1,-2\le x\le2\bigr\}.
\end{align*}
See, e.g., \cite[Example~1.5.6]{AA02}.  
\end{example}

A bounded strongly Lipschitz domain is bilipschitz equivalent to a smooth domain
in the following sense. The proof of this fact is given in the 
Appendix~\ref{sec:appendix}. 

\begin{proposition}
\label{prop:lip-sim-smooth}
Let $\Omega\subset{\mathbb{R}}^n$ be a bounded strongly Lipschitz domain. 
Then there exists a bilipschitz map $\phi:{\mathbb{R}}^n\to{\mathbb{R}}^n$ 
where $\phi^{-1}(\Omega)=\Omega'$ is a smooth domain in ${\mathbb{R}}^n$
satisfying $\phi({\mathbb{R}}^n\setminus\overline{\Omega'})
={\mathbb{R}}^n\setminus\overline{\Omega}$ and $\phi(\partial\Omega')
=\phi(\partial\Omega)$.
\end{proposition}

We now take the property of weakly Lipschitz domains spelled out in 
Remark~\ref{rem:weakLip} (though without the condition that the bilipschitz maps 
extend to slightly larger sets) as our definition of very weakly Lipschitz domains, 
because this is all that is needed in proving many of our results.

\begin{definition}
\label{def:awLip}
We call  an open set
$\Omega\subset{\mathbb{R}}^n$ a {\tt very weakly Lipschitz domain} provided it satisfies the property \eqref{eq:Omega} below:
\begin{equation}
\label{eq:Omega}\tag{VWL}
\begin{array}{l}
\mbox{there exist }\Bigl(\rho_j:B\to\Omega_j\Bigr)_{j=1,\dots,M} 
\mbox{ bilipschitz maps such that }
\displaystyle{\Omega=\bigcup_{j=1}^M\Omega_j},
\\
\mbox{and for each }j=1,\dots,M, \mbox{ there exists a Lipschitz function } 
\chi_j:\Omega\to[0,1] 
\\
\mbox{such that }{\rm sppt}_\Omega\chi_j\subset \Omega_j\mbox{ and }
\displaystyle{\sum_{j=1}^M \chi_j(x)=1}\mbox{ for all }x\in\Omega.
\end{array}
\end{equation}
\end{definition}

\begin{example}
\label{ex:verywLipNOTwLip}
Let us reconsider the domain $\Omega$ of Example~\ref{ex:LipNOTgloballyLip}. 
It is not weakly Lipschitz because its boundary does not form an interface 
between $\Omega$ and ${\mathbb{R}}^n\setminus\overline\Omega$. However 
it is very weakly Lipschitz (with $M=1$ and $\chi_1=1$) as can be shown as 
follows. Set
$$
\Omega':=\bigl\{(x,y)\in{\mathbb{R}}^2\,;\,0<x^2+y^2<1, |\arg(x,y)|<
\tfrac{\pi}{2}\bigr\}
$$
and define $\phi:\Omega\to\Omega'$ by
$\phi(x,y):= \bigl(r\cos(\frac{\theta}{2}),r\sin(\frac{\theta}{2})\bigr)$ where 
$r:=(x^2+y^2)^{\frac{1}{2}}$ and $\theta=\arg(x,y)$. Now $\phi$ is a 
bilipschitz map from $\Omega$ to $\Omega'$ in the sense of 
Definition~\ref{def:bilip}, and $\Omega'$ is bilipschitz equivalent to a ball, so 
that $\Omega$ is bilipschitz equivalent to a ball. 
\end{example} 

\subsection{Differential forms} 

We consider the {\tt exterior derivative} $d:=\nabla\wedge=
\sum_{j=1}^n \partial_j e_j\wedge$ and the {\tt interior deri\-va\-tive} 
(or co-derivative) 
$\delta:=-\nabla\lrcorner\,=-\sum_{j=1}^n \partial_j e_j\lrcorner\,$ acting on 
{\tt differential forms} on a domain $\Omega\subset{\mathbb{R}}^n$, i.e. 
acting on functions from $\Omega$ to  the exterior algebra 
$\Lambda=\Lambda^0\oplus\Lambda^1\oplus\dots\oplus\Lambda^n$ of
${\mathbb{R}}^n$. 

We denote by $\bigl\{e_S\,;\,S\subset\{1,\dots,n\}\bigr\}$ the basis
for $\Lambda$. The space of $\ell$-vectors $\Lambda^\ell$ is the
span of $\bigl\{e_S\,;\,|S|=\ell\bigr\}$, where 
$$
e_S=e_{j_1}\wedge e_{j_2}\wedge\dots\wedge e_{j_\ell}\quad
\mbox{for}\quad S=\{e_{j_1},\dots,e_{j_\ell}\} \quad\text{with }\ j_1<j_2<\dots<j_{\ell}.
$$
Remark that $\Lambda^0$, the space of complex scalars, is the span of 
$e_\emptyset$ ($\emptyset$ being the empty set). We set 
$\Lambda^\ell=\{0\}$ if $\ell<0$ or $\ell>n$.

On the exterior algebra $\Lambda$, the basic operations are
\begin{enumerate}[$(i)$ ]
\item 
the exterior product $\wedge:\Lambda^k\times\Lambda^\ell\to\Lambda^{k+\ell}$,
\item
the interior product $\lrcorner\,:\Lambda^k\times\Lambda^\ell\to\Lambda^{\ell-k}$,
\item
the Hodge star operator $\star:\Lambda^\ell\to\Lambda^{n-\ell}$, 
\item
the inner product $\langle\cdot,\cdot\rangle:\Lambda^\ell\times\Lambda^\ell
\to{\mathbb{R}}$. 
\end{enumerate}
If $a\in\Lambda^1$, $u\in\Lambda^\ell$ and $v\in\Lambda^{\ell+1}$, then
$$
\langle a\wedge u,v\rangle=\langle u,a\lrcorner\,v\rangle.
$$
For more details, we refer to, e.g., \cite[Section~2]{AMcI03} and 
\cite[Section~2]{CMcI10}, noting that both these papers contain some historical 
background (and being careful that $\delta$ has the opposite sign in 
\cite{AMcI03}). In particular, we note the relation between $d$ and
$\delta$ via the Hodge star operator:
\begin{equation}
\label{eq:*d=delta}
\star\delta u=(-1)^\ell d(\star\,u) \quad\mbox{and}\quad
\star du=(-1)^{\ell-1}\delta(\star\,u) \quad\mbox{for an $\ell$-form }u.  
\end{equation}
The domains of the differential operators $d$ and $\delta$, denoted by 
${\rm{\sf D}}(d,\Omega)$ and ${\rm{\sf D}}(\delta,\Omega)$, or more
simply ${\rm{\sf D}}(d)$ and ${\rm{\sf D}}(\delta)$, are defined by
$$
{\rm{\sf D}}(d):=\bigl\{u\in L^2(\Omega,\Lambda); du\in L^2(\Omega,\Lambda)\bigr\}
\quad \mbox{and}\quad
{\rm{\sf D}}(\delta):=\bigl\{u\in L^2(\Omega,\Lambda); 
\delta u\in L^2(\Omega,\Lambda)\bigr\}.
$$
Similarly, the $L^p$ versions of these domains read
$$
{\rm{\sf D}}^p(d,\Omega):=\bigl\{u\in L^p(\Omega,\Lambda); 
du\in L^p(\Omega,\Lambda)\bigr\}
\  \mbox{ and }\  
{\rm{\sf D}}^p(\delta,\Omega):=\bigl\{u\in L^p(\Omega,\Lambda); 
\delta u\in L^p(\Omega,\Lambda)\bigr\}.
$$
The differential operators $d$ and $\delta$ satisfiy $d^2=d\circ d=0$ 
and $\delta^2=\delta\circ\delta=0$.
We will also consider the adjoints of $d$ and $\delta$ in the sense of
maximal adjoint operators in a Hilbert space: $\underline{\delta}:=d^*$
and $\underline{d}:=\delta^*$. They are defined as the closures in 
$L^2(\Omega,\Lambda)$ of the closable
operators $\bigl(d^*,{\mathscr{C}}_c^\infty(\Omega,\Lambda)\bigr)$ and
$\bigl(\delta^*,{\mathscr{C}}_c^\infty(\Omega,\Lambda)\bigr)$.
The next result was proved in \cite[Corollary~4.4]{AMcI03}.

\begin{proposition}
\label{d,delta_in_weakLip} 
In the case where $\Omega$ is a bounded weakly Lipschitz domain,
the operators $d^*=\underline{\delta}$ and $\delta^*=\underline{d}$ have 
the following representation
\begin{align*}
{\rm{\sf D}}(\underline{d},\Omega)={\rm{\sf D}}(\underline{d})
:&=\bigl\{u\in L^2(\Omega,\Lambda); d\tilde{u}
\in L^2({\mathbb{R}}^n,\Lambda)\bigr\},\quad 
\underline{d}u=(d\tilde{u})_{|_\Omega} \mbox{ for }
u\in {\rm{\sf D}}(\underline{d}),
\\
{\rm{\sf D}}(\underline{\delta},\Omega)={\rm{\sf D}}(\underline{\delta})
:&=\bigl\{u\in L^2(\Omega,\Lambda); 
\delta\tilde{u}\in L^2({\mathbb{R}}^n,\Lambda)\bigr\}, \quad 
\underline{\delta}u=(\delta \tilde{u})_{|_\Omega}
\mbox{ for }u\in {\rm{\sf D}}(\underline{\delta}).
\end{align*}
where $\tilde{u}$ denotes the zero-extension of $u$ to ${\mathbb{R}}^n$.
\end{proposition}

A well-known property of the differential operator $d$ is that it commutes
with a change of variables as stated below, see, e.g., \cite[Definition~1.2.1 and
Proposition~1.2.2]{AA02}.

\begin{definition}
\label{def:changevar}  
Let $\Omega$ be an open set in ${\mathbb{R}}^n$ and
$\rho:\Omega\to\rho(\Omega)$ a bilipschitz transformation. 
Denote by $J_\rho(y)$ the Jacobian matrix of $\rho$ at a point $y\in \Omega$ 
and extend it to an isomorphism $J_\rho(y):\Lambda\to\Lambda$ such that
$$
J_\rho(y)(e_{i_1}\wedge\dots\wedge e_{i_k})=(J_\rho(y)e_{i_1})\wedge\dots
\wedge (J_\rho(y)e_{i_k}), \quad \{i_1,\dots,i_k\}\subset\{1,\dots,n\}.
$$
The {\tt pullback} of a field $u:\rho(\Omega)\to \Lambda$ is denoted by  
$\rho^*u:\Omega\to \Lambda$, the {\tt push forward} of a field 
$f:\Omega\to\Lambda$ by $\rho_*f:\rho(\Omega)\to\Lambda$ and 
$\tilde\rho_*^{-1}u:={\rm Jac}(\rho)\rho_*^{-1}u:\Omega\to\Lambda$ are 
defined by
$$
(\rho^*u)(y):=J_\rho(y)^*\bigl(u(\rho(y)\bigr)\quad\mbox{and}\quad
(\rho_*^{-1}u)(y):=J_\rho(y)^{-1}\bigl(u(\rho(y)\bigr),
\quad y\in B,
$$
and where ${\rm Jac}(\rho)(y)$ denotes the Jacobian determinant of $\rho$ at a
point $y\in \Omega$.
\end{definition}

\begin{remark}
\label{rem:bounds}
Note that for all $p\in[1,\infty]$, 
$\rho^*:L^p(\rho(\Omega),\Lambda)\to L^p(\Omega,\Lambda)$ and
$(\rho_*)^{-1}:L^p(\rho(\Omega),\Lambda)\to L^p(\Omega,\Lambda)$ are 
bounded with norms controlled by 
$\displaystyle{\esssup_{y\in \Omega}\|J_\rho(y)\|_{{\mathscr{L}}(\Lambda)}}$ and
$\displaystyle{\esssup_{y\in \Omega}\|J_\rho(y)^{-1}\|_{{\mathscr{L}}(\Lambda)}}$, 
and hence by the Lipschitz constants of $\rho$ and $\rho^{-1}$.
\end{remark} 

\begin{remark}
\label{rem:drho=rhod}
For $\rho$ as in Definition~\ref{def:changevar} and a field 
$u:\rho(\Omega)\to\Lambda$ the following commutation properties hold:
\begin{equation}
\label{eq:com}
d(\rho^*u)=\rho^*(d u)\quad\mbox{and}\quad
\delta(\tilde\rho_*^{-1}u)=\tilde\rho_*^{-1}(\delta u).
\end{equation}
In particular, if $u\in {\rm{\sf D}}(d,\rho(\Omega))$, then 
$\rho^*u\in{\rm{\sf D}}(d,B)$ and if $u\in {\rm{\sf D}}(\delta,\rho(\Omega))$, then 
$\tilde\rho_*^{-1}u\in{\rm{\sf D}}(\delta,\Omega)$.

We also have the following homomorphism properties:
$$
\begin{array}{lcl}
\rho^*(u\wedge v)=\rho^*u\wedge\rho^*v, &\quad&
\rho_*^{-1}(u\wedge v)=\rho_*^{-1}u\wedge\rho_*^{-1}v,
\\[4pt]
\rho^*(u\lrcorner\, v)=\rho_*^{-1}u\lrcorner\,\rho^*v,
&&
\rho_*^{-1}(u\lrcorner\, v)=\rho^*u\lrcorner\,\rho_*^{-1}v.
\end{array}
$$
\end{remark}

\begin{remark}
\label{product-rule}
By the product rule for the exterior derivative and the interior derivative we 
have that for all bounded Lipschitz scalar-valued  functions $\eta$, for all 
$u\in {\rm{\sf D}}^p(d,\Omega)$ and $v\in{\rm{\sf D}}^p(\underline{\delta},\Omega)$, 
then $\eta u\in {\rm{\sf D}}^p(d,\Omega)$, 
$\eta v\in{\rm{\sf D}}^p(\underline{\delta},\Omega)$ with
\begin{equation}
\label{eq:product-rule}
d(\eta u)=\eta\,du+\nabla\eta\wedge u\quad\mbox{and}\quad
\delta(\eta v)=\eta\,\delta u-\nabla\eta\lrcorner\,v .
\end{equation}
More generally, for $u$ a bounded Lipschitz $\ell$-form, for all 
$v\in {\rm{\sf D}}^p(d,\Omega)$, it holds
\begin{equation}
\label{eq:product-rule2}
d(u\wedge v)=du\wedge v+(-1)^\ell u\wedge dv,
\end{equation}
which gives also for all bounded Lipschitz scalar-valued  functions $\eta$, 
and for all $u\in {\rm{\sf D}}^p(d,\Omega)$:
\begin{equation}
\label{eq:product-rule3}
d(\nabla\eta\wedge u)=-\nabla\eta\wedge du.
\end{equation}
\end{remark} 

\subsection{Bisectoriality, sectoriality and functional calculus}
\label{subsec:bisect-fc}

\begin{definition}
\label{def:bisect}
A closed unbounded operator $A$ on a Banach space $X$ is said to be 
{\tt bisectorial} of angle $\omega\in\bigl[0,\frac{\pi}{2}\bigr)$ if the spectrum of 
$A$ is contained in the double sector $S_\omega$ and for all 
$\theta\in\bigl(\omega,\frac{\pi}{2}\bigr)$, the following resolvent estimate holds: 
$$
\sup_{z\in{\mathbb{C}}\setminus S_\theta}
\|({\rm I}+zA)^{-1}\|_{{\mathscr{L}}(X)}<\infty.
$$
\end{definition}

\begin{remark}
\label{rem:fcPsi}
Let $\mu\in\bigl(0,\frac{\pi}{2}\bigr)$.
Denote by $\Psi(S_\mu^\circ)$ the subspace of continuous functions
$f:S_\mu\to{\mathbb{C}}$ holomorphic on $S_\mu^\circ$ for which there 
exists $s>0$ such that $\displaystyle{\sup_{z\in S_\mu^\circ}\Bigl\{
\frac{|z|^s|f(z)|}{1+|z|^{2s}}\Bigr\}<\infty}$. 
Let $A$ be a bisectorial operator of angle $\omega\in[0,\mu)$
on a Banach space $X$. For all $f\in \Psi(S_\mu^\circ)$, we can define
for $\theta\in (\omega,\mu)$ 
$$
f(A)u:=\tfrac1{2\pi i}\int_{\partial S_\theta^\circ}f(z)(z{\rm I}\,-A)^{-1}u\,{\rm d}z,
$$ 
where the boundary of the double sector $\partial S_\theta^\circ$ is oriented 
counterclockwise.
Note that the integral above converges in norm thanks to the definition of
functions belonging to $\Psi(S_\mu^\circ)$ and the estimate on the 
resolvents of $A$.
\end{remark}

\begin{definition}
\label{def:functcalc}
Let $0\le\omega<\mu<\frac{\pi}{2}$.
A bisectorial operator $A$ of angle $\omega$ on a Banach space $X$ is 
said to admit a {\tt bounded $S_\mu^\circ$ holomorphic functional calculus} in $X$ if 
for $\theta\in(\omega,\mu)$ there exists a constant $K_\theta>0$ such that for all 
$f\in \Psi(S_\mu^\circ)$, we have that
$$
\|f(A)\|_{{\mathscr{L}}(X)}\le K_\theta\|f\|_{L^\infty(S_\theta)}.
$$
\end{definition}

\begin{remark} 
\label{self-adj-fun-calc}
Every self-adjoint operator $S$ in a Hilbert space $X$ is bisectorial of angle 0 
with resolvent estimate $\displaystyle{\sup_{z\in{\mathbb{C}}\setminus S_\theta}
\|({\rm I}+zS)^{-1}\|_{{\mathscr{L}}(X)}\le \frac{1}{\sin \theta}}$, 
and has a bounded holomorphic functional calculus with $K_\theta =1$.
See, e.g., \cite{McI86}.
\end{remark}

The results above can be adapted to the case of sectorial
operators suited for second order differential operators.

\begin{definition}
\label{def:sect}
A closed unbounded operator $A$ on a Banach space $X$ is said to be 
{\tt sectorial} of angle $\omega\in [0,\pi)$ if the spectrum of 
$A$ is contained in the sector $S_{\omega+}$ and for all 
$\theta\in\bigl(\omega,\pi\bigr)$, the following resolvent estimate holds: 
$$
\sup_{z\in{\mathbb{C}}\setminus S_{\theta+}}
\|({\rm I}+zA)^{-1}\|_{{\mathscr{L}}(X)}<\infty.
$$
\end{definition}

\begin{remark}
\label{rem:fcPsi-sect}
Let $\mu\in(0,\pi)$. As before,
denote by $\Psi(S_{\mu+}^\circ)$ the subspace of continuous functions
$f:S_{\mu+}\to{\mathbb{C}}$, holomorphic on $S_{\mu+}^\circ$ for which 
there exists $s>0$
such that $\displaystyle{\sup_{z\in S_{\mu+}^\circ}\Bigl\{
\frac{|z|^s|f(z)|}{1+|z|^{2s}}\Bigr\}<\infty}$. 
Let $A$ be a sectorial operator of angle $\omega\in[0,\mu)$
on a Banach space $X$. For all $f\in \Psi(S_{\mu+}^\circ)$, we can define
for $\theta\in (\omega,\mu)$ 
$$
f(A)u:=\tfrac1{2\pi i}\int_{\partial S_{\theta+}^\circ}f(z)(z{\rm I}\,-A)^{-1}u\,{\rm d}z,
$$ 
where the boundary of the sector $\partial S_{\theta+}^\circ$ is oriented 
counterclockwise.
Note that the integral above converges in norm thanks to the definition of
functions belonging to $\Psi(S_{\mu+}^\circ)$ and the estimate on the 
resolvents of $A$.
\end{remark}

\begin{definition}
\label{def:functcalcsect}
Let $0\le\omega<\mu<\frac{\pi}{2}$.
A sectorial operator $A$ of angle $\omega$ on a Banach space $X$ is 
said to admit a {\tt bounded $S_{\mu+}^\circ$ holomorphic functional calculus} 
in $X$ if for $\theta\in(\omega,\mu)$ there exists a constant $K_\theta>0$ such 
that for all $f\in \Psi(S_{\mu+}^\circ)$, we have that
$$
\|f(A)\|_{{\mathscr{L}}(X)}\le K_\theta\|f\|_{L^\infty(S_{\theta+})}.
$$
\end{definition}

\begin{definition}
\label{def:OD}
Let $\Omega\subset{\mathbb{R}}^n$ be an open set and let $q\in[1,\infty)$. 
A family of bounded operators $\bigl\{R_z, z\in Z\bigr\}$ (where 
$Z\subset{\mathbb{C}}$) on $L^q(\Omega)$ is said to admit 
(exponential) {\tt off-diagonal bounds} $L^q-L^q$ (of first order) if there exists 
$C,c>0$ such that for all $E,F\subset{\mathbb{R}}^n$ Borel sets, we have that
$$
\bigl\|1\!{\rm l}_ER_z1\!{\rm l}_Fu\bigl\|_{L^q(\Omega)}
\le C e^{-c\,\frac{{\rm dist}\,(E,F)}{|z|}}\|u\|_{L^q(\Omega)},
\quad \forall\, z\in Z,\ \forall\, u\in L^q(\Omega).
$$
\end{definition}

\begin{remark}
\label{rem:OD}
If a family of bounded operators $\bigl\{R_z, z\in Z\bigr\}$ on $L^q(\Omega)$ 
admits off-diagonal bounds $L^q-L^q$, then the family of adjoints 
$\bigl\{{R_z}^*, z\in Z\bigr\}$ admits off-diagonal bounds $L^{q'}-L^{q'}$.
\end{remark}

\section{Hodge-Dirac operators}
\label{HDops}

\begin{definition}
\label{def:HodgeDirac}
\begin{enumerate}[$(i)$ ]
\item
The {\tt Hodge-Dirac operator} on $\Omega$ with {\tt normal boundary 
con\-di\-tions} is
$$
D_{{}^\bot}:=\delta^*+\delta=\underline{d}+\delta.
$$
Note that $-{\Delta_{{}^\bot}}:={D_{{}^\bot}}^2
=\underline{d}\delta+\delta\underline{d}$ is
the Hodge-Laplacian with relative (generalised Dirichlet) boundary 
conditions.

For a scalar function $u:\Omega\to\Lambda^0$ we have that 
$-{\Delta_{{}^\bot}}u=\delta\underline{d}u=-\Delta_Du$, where $\Delta_D$ is 
the Dirichlet Laplacian.
\item
The {\tt Hodge-Dirac operator} on $\Omega$ with {\tt tangential boundary 
conditions} is
$$
D_{{}^\Vert}:=d+d^*=d+\underline{\delta}.
$$
Note that $-{\Delta_{{}^\Vert}}:={D_{{}^\Vert}}^2
=d\underline{\delta}+\underline{\delta}d$ is 
the Hodge-Laplacian with absolute (generalised Neumann) boundary 
conditions.

For a scalar function 
$u:\Omega\to\Lambda^0$ we have that 
$-{\Delta_{{}^\Vert}}u=\underline{\delta}du=-\Delta_Nu$, where 
$\Delta_N$ is the Neumann Laplacian.
\end{enumerate}
\end{definition}

Following \cite[Section 4]{AKMcI06Invent}, we have that the operators 
${D_{{}^\bot}}$ and ${D_{{}^\Vert}}$ are closed densely defined operators 
in $L^2(\Omega,\Lambda)$, and that 
\begin{align*}
L^2(\Omega,\Lambda)&=\overline{{\rm{\sf R}}(d)}\stackrel{\bot}{\oplus}
\overline{{\rm{\sf R}}(\underline{\delta})}\stackrel{\bot}
{\oplus}{\rm{\sf N}}(D_{{}^\Vert})
\\
&=\overline{{\rm{\sf R}}(\delta)}\stackrel{\bot}{\oplus}
\overline{{\rm{\sf R}}(\underline{d})} \stackrel{\bot}{\oplus}{\rm{\sf N}}(D_{{}^\bot}),
\end{align*}
where ${\rm{\sf N}}(D_{{}^\Vert})= {\rm{\sf N}}(d)\cap{\rm{\sf N}}(\underline{\delta})
={\rm{\sf N}}\bigl({\Delta_{{}^\Vert}}\bigr)$ and
${\rm{\sf N}}(D_{{}^\bot})={\rm{\sf N}}(\delta)\cap{\rm{\sf N}}(\underline{d})
={\rm{\sf N}}\bigl({\Delta_{{}^\bot}}\bigr)$

\begin{remark}
\label{rem:Hodge-dec-(H)}
If $\Omega$ satisfies \eqref{eq:Omega}, then it is essentially proved in 
\cite[proof of Theorem~1.3, $(i)$  p.\,19-20]{AMcI03}, that ${\rm{\sf R}}(d)$
and ${\rm{\sf R}}(\underline{\delta})$, as well as ${\rm{\sf R}}(\delta)$ and
${\rm{\sf R}}(\underline{d})$, are closed subspaces of $L^2(\Omega,\Lambda)$
and that ${\rm{\sf N}}(D_{{}^\Vert})={\rm{\sf N}}\bigl({\Delta_{{}^\Vert}}\bigr)$ and
${\rm{\sf N}}(D_{{}^\bot})={\rm{\sf N}}\bigl({\Delta_{{}^\bot}}\bigr)$ are finite 
dimensional. We shall include a proof of these facts in Section \ref{sec:potentials}.
\end{remark}

\begin{definition}
\label{def:projections}
The {\tt Hodge decompositions} from Remark~\ref{rem:Hodge-dec-(H)} are
accompanied with the orthogonal projections
\begin{align*}
&{\mathcal{P}}_{{\rm{\sf R}}(d)}:L^2(\Omega,\Lambda)\to \overline{{\rm{\sf R}}(d)},
\quad
{\mathcal{P}}_{{\rm{\sf R}}(\underline{\delta})}:L^2(\Omega,\Lambda)\to 
\overline{{\rm{\sf R}}(\underline{\delta})},\quad
{\mathcal{P}}_{{\rm{\sf N}}(D_{{}^\Vert})}:L^2(\Omega,\Lambda)\to
{\rm{\sf N}}(D_{{}^\Vert});
\\
&{\mathcal{P}}_{{\rm{\sf R}}(\delta)}:L^2(\Omega,\Lambda)\to 
\overline{{\rm{\sf R}}(\delta)},
\quad {\mathcal{P}}_{{\rm{\sf R}}(\underline{d})}:L^2(\Omega,\Lambda)\to 
\overline{{\rm{\sf R}}(\underline{d})},\quad
{\mathcal{P}}_{{\rm{\sf N}}(D_{{}^\bot})}:L^2(\Omega,\Lambda)\to 
{\rm{\sf N}}(D_{{}^\bot}).
\end{align*}
Moreover, noting that $d:{\rm{\sf D}}(d)\cap\overline{{\rm{\sf R}}(\underline{\delta})}
\to \overline{{\rm{\sf R}}(d)}$ is one-to-one we define 
$$
\underline{R}:L^2(\Omega,\Lambda)\to {\rm{\sf R}}(\underline{\delta}),\quad
\begin{cases}
d\underline{R} u=u\mbox{ if }u\in \overline{{\rm{\sf R}}(d)},\\
\underline{R} u=0\mbox{ if } 
u\in \overline{{\rm{\sf R}}(\underline{\delta})}\stackrel{\bot}{\oplus}
{\rm{\sf N}}(D_{{}^\Vert}).
\end{cases}
$$
In particular, we have that
$$
{\rm I}\,=d\underline{R}+\overline{\underline{R}d}
+{\mathcal{P}}_{{\rm{\sf N}}(D_{{}^\Vert})}.
$$
Note that $\underline{R}$ is a potential operator, in the sense that, if
$u\in{\rm{\sf R}}(d)$ then $u=df$ where $f=\underline{R}u$.
\end{definition}

\begin{remark}
\label{rem:smooth-or-convex}
If the domain $\Omega\subset{\mathbb{R}}^n$ is convex or of class 
${\mathscr{C}}^{1,1}$, we have that
${\rm{\sf D}}(D_{{}^\bot}),{\rm{\sf D}}(D_{{}^\Vert})\subset H^1(\Omega,\Lambda)$
(see \cite[Theorems~2.9, 2.12, 2.17]{ABDG98} for the proof in dimension $n=3$,
\cite[Theorem~4.10 and Remark~4.11]{AMcI03}).
This is however not true in general.
If $\Omega$ is a strongly Lipschitz domain, then it can be proved that
${\rm{\sf D}}(D_{{}^\bot}),{\rm{\sf D}}(D_{{}^\Vert})\subset 
H^{\frac{1}{2}}(\Omega,\Lambda)$ as shown in \cite{Co90} in dimension~3
and \cite[Theorem~11.2]{MMT01} in arbitrary dimension (see also the estimate
\eqref{eq:solvability} below). 
\end{remark}

\begin{remark} 
\label{rem:D||,DT}
At this point we remark that the theory concerning the Hodge-Dirac operator 
with normal boundary conditions, $D_{{}^\bot}=\delta^*+\delta=\underline{d}+\delta$, 
is entirely analogous to the theory concerning the Hodge-Dirac operator with 
tangential boundary conditions, $D_{{}^\Vert}=d+d^*=d+\underline{\delta}$.
Either the proofs for one can be mimicked for the other, or the results for one 
can be obtained form the results for the other by the Hodge star operator and 
appropriate changes of sign. So from now on we will state our results for $d$, 
$\underline{\delta}$ and $D_{{}^\Vert}$, noting here that corresponding results 
hold for $\delta$, $\underline{d}$ and $D_{{}^\bot}$.
\end{remark}

\section{Potential operators on very weakly Lipschitz domains}
\label{sec:potentials}

The unit ball $B=B(0,1)$  in ${\mathbb{R}}^n$ is starlike with respect to the ball 
$\frac{1}{2}B:=B(0,\frac{1}{2})$.
For $p\in(1,\infty)$ and $s\in{\mathbb{R}}$, let 
$R_B:W^{s-1,p}(B,\Lambda)\to W^{s,p}(B,\Lambda)$
be a Poincar\'e-type map (relative to a non negative smooth function 
$\theta\in{\mathscr{D}}:={\mathscr{C}}_c^\infty(B)$ with support in 
$\frac{1}{2}B$ and $\int\theta=1$) as defined in 
\cite[Definition~3.1 and (3.9)]{CMcI10} (building on \cite{MMM08}; see 
also \cite{Bo79}) in the case of domains which are starlike with respect to a ball. 
In those papers a theory of potential operators in Sobolev spaces on strongly 
Lipschitz domains was developed. In this section we follow some of the 
techniques developed there to consider a somewhat different context, namely 
potential operators mapping $L^p(\Omega,\Lambda)$ to 
$L^{p^S}(\Omega,\Lambda)$ on very weakly Lipschitz domains.

The operator $R_B$ has the following representation 
\begin{align}
\label{eq:defRB}
&R_Bf_\ell(y):=\int_{\frac{1}{2}B}\theta(a)(y-a)\lrcorner\,\Bigl(\int_0^1t^{\ell-1}
f_\ell(a+t(y-a))\,{\rm d}t\Bigr)\,{\rm d}a\\
&\mbox{for an $\ell$-form }f_\ell\ (\ell=1,\dots,n) 
\nonumber
\end{align}
($R_Bf_0=0$) and satisfies
\begin{equation}
\label{eq:Rd+dR=I-K}
R_B d\,f+d\, R_Bf=f-K_Bf\quad
\mbox{where}\quad K_Bf={}_{{\mathscr{D}}}\langle\theta,
f_0\rangle_{{\mathscr{D}}'}\, e_\emptyset
\end{equation}
for all $f=f_0+f_1+\cdots f_n\in W^{s,p}(B,\Lambda)=
W^{s,p}(B,\Lambda^0)\oplus W^{s,p}(B,\Lambda^1)\oplus\cdots\oplus 
W^{s,p}(B,\Lambda^n)$, where ${}_{{\mathscr{D}}}\langle\cdot,
\cdot\rangle_{{\mathscr{D}}'}$ denotes the duality pairing between 
${\mathscr{D}}$ and ${\mathscr{D}}'$.
The operator $K_B$ is infinitely smoothing in the sense that 
for all $f\in{\mathscr{D}}'$, $K_Bf\in{\mathscr{C}}^\infty(B,\Lambda)$.
Moreover, $K_Bf=0$ if $f=dg$ for $g\in{\rm{\sf D}}(d,B)$, which implies that
the operator $R_B$ is a true potential for $d$ on $B$ in the sense that 
for all $p\in(1,\infty)$
\begin{equation}
\label{eq:dR}
\mbox{if } f\in{\rm{\sf R}}^p(d,B),\mbox{ then }f=d\,R_Bf.
\end{equation}
The mapping properties of $R_B$ imply in particular that, 
\begin{equation}
\label{eq:est-dR}
dR_B:L^p(B,\Lambda)\to L^p(B,\Lambda),\quad \forall p\in(1,\infty),
\end{equation}
so that $dR_B$ is a projection from $L^p(B,\Lambda)$ onto ${\rm{\sf R}}^p(d,B)$.
We also have that for $p\in(1,\infty)$, the adjoint operator of $R_B$, ${R_B}^*$, 
maps $L^p(B,\Lambda)$ to 
$W^{1,p}_{\overline{B}}(\Lambda)\hookrightarrow L^{p^S}(B,\Lambda)$
where $p^S$ is as in Notation~\ref{not:p,p',pS,p*}. Therefore, thanks to 
Remark~\ref{rem:qS'}, we have that
\begin{equation}
\label{eq:R:Lq->LqS}
R_B:L^p(B,\Lambda)\to L^{p^S}(B,\Lambda)\cap {\textsf{D}}^p(d,B)
\end{equation}
and
\begin{equation}
\label{eq:K:Lp->L8}
K_B:L^p(B,\Lambda)\to L^\infty(B,\Lambda)\cap{\textsf{D}}^p(d,B)
\end{equation}
are bounded for all $p\in(1,\infty)$. Since the range of $K_B$ is one-dimensional, 
the operator $K_B$ is compact in $L^p(B,\Lambda)$ for every $p\in(1,\infty)$. 

Let now $\Omega\subset{\mathbb{R}}^n$ satisfy property \eqref{eq:Omega}:
$\Omega=\cup_{j=1}^M\rho_j(B)$ with $\chi_j:\Omega\to[0,1]$
Lipschitz functions such that ${\rm sppt}_\Omega\,\chi_j\subset \rho_j(B)$ 
and $\sum_{j=1}^M\chi_j=1$ on $\Omega$. Following the construction
of \cite{CMcI10} we define for $u\in L^p(\Omega,\Lambda)$
$$
\tilde R_\Omega u=\sum_{j=1}^M\chi_j (\rho_j^*)^{-1} R_B(\rho_j^*u).
$$
By Remark \ref{rem:bounds},   
$\tilde R_\Omega:L^p(\Omega,\Lambda)\to L^{p^S}(\Omega,\Lambda)
\cap {\textsf{D}}^p(d,B)$ for all $q\in (1,\infty)$. Moreover, for all 
$u\in {\rm{\sf D}}^p(d,\Omega)$ we have, thanks to the product rule 
\eqref{eq:product-rule}, the commutation
property \eqref{eq:com} and the relation \eqref{eq:Rd+dR=I-K} satisfied by $R_B$,
that
\begin{align*}
d\tilde R_\Omega u&=
\sum_{j=1}^M\chi_j d\bigl[(\rho_j^*)^{-1} R_B(\rho_j^*u)\bigr]
+\sum_{j=1}^M\nabla\chi_j\wedge\bigl[(\rho_j^*)^{-1} R_B(\rho_j^*u)\bigr]
\\
&=\sum_{j=1}^M\chi_j \bigl[(\rho_j^*)^{-1} d R_B(\rho_j^*u)\bigr]
+\sum_{j=1}^M\nabla\chi_j\wedge\bigl[(\rho_j^*)^{-1} R_B(\rho_j^*u)\bigr]
\\
&=\sum_{j=1}^M\chi_j \bigl[(\rho_j^*)^{-1} ({\rm I}-K_B-R_Bd)(\rho_j^*u)\bigr]
+\sum_{j=1}^M\nabla\chi_j\wedge\bigl[(\rho_j^*)^{-1} R_B(\rho_j^*u)\bigr]
\\
&=u-\tilde R_\Omega d u-\tilde K_\Omega u
\end{align*}
where
$$
\tilde K_\Omega u=\sum_{j=1}^M\Bigl(\chi_j(\rho_j^*)^{-1}K_B(\rho_j^*u)
-\nabla\chi_j\wedge\bigl[(\rho_j^*)^{-1} R_B(\rho_j^*u)\bigr]\Bigr).
$$
The operator $\tilde K_\Omega$ is compact in $L^p(\Omega,\Lambda)$ for
all $p\in(1,\infty)$; it is indeed a sum of compositions of bounded operators 
($\rho_j^*$, $(\rho_j^*)^{-1}$ and multiplication with $\chi_j$ or $\nabla\chi_j$) with
compact operators ($K_B$ and $R_B$). 

The relation $d\tilde R_\Omega +\tilde R_\Omega d={\rm I}\,-\tilde K_\Omega$
on ${\rm{\sf D}}^p(d,\Omega)$ implies directly that $\tilde K_\Omega$ 
commutes with $d$ on ${\rm{\sf D}}^p(d,\Omega)$. Moreover,
thanks to the mapping properties of $R_B$ and $K_B$, it is clear that
$\tilde K_\Omega$ maps $L^q(\Omega,\Lambda)$ to $L^{q^S}(\Omega,\Lambda)$
for all $q\in (1,\infty)$. It is also obvious that $\tilde K_\Omega$ maps 
$L^p(\Omega,\Lambda)$ to ${\textsf{D}}^p(d,\Omega)$ thanks to the mapping
properties of $R_B$ and $K_B$, the commutation property \eqref{eq:com}
and the product rules \eqref{eq:product-rule} and \eqref{eq:product-rule3}.
Therefore, we see that $\tilde K_\Omega\phantom{}^n$ maps 
$L^p(\Omega,\Lambda)$ to 
$L^\infty(\Omega,\Lambda)\cap {\textsf{D}}^p(d,\Omega)$ for all $p>1$.
We define the following operators $\tilde{\tilde R}_\Omega$ and 
$\tilde{\tilde K}_\Omega$:
$$
\tilde{\tilde R}_\Omega:=\bigl({\rm I}+{\tilde K}_\Omega
+{\tilde K_\Omega}\phantom{}^{2}+\dots
+{\tilde K_\Omega}\phantom{}^{n-1}\bigr)\tilde R_\Omega\quad\mbox{and}\quad
\tilde{\tilde K}_\Omega:={\tilde K_\Omega}\phantom{}^n.
$$
It follows that $\tilde{\tilde K}_\Omega$ is compact in $L^p(\Omega,\Lambda)$
for all $p\in(1,\infty)$ (as a composition of compact operators) and
$$
\begin{array}{ll}
\tilde{\tilde R}_\Omega:L^p(\Omega,\Lambda)\to L^{p^S}(\Omega,\Lambda)\cap
{\textsf{D}}^p(d,\Omega),\quad
&\forall\, p\in(1,\infty)
\\[4pt]
\tilde{\tilde K}_\Omega:L^p(\Omega,\Lambda)\to L^\infty(\Omega,\Lambda)\cap
{\textsf{D}}^p(d,\Omega),\quad
&\forall\, p\in(1,\infty),
\\[4pt]
d\tilde{\tilde R}_\Omega+\tilde{\tilde R}_\Omega d={\rm I}\,-\tilde{\tilde K}_\Omega, 
&d\tilde{\tilde K}_\Omega=\tilde{\tilde K}_\Omega d \qquad\mbox{on }
{\rm{\sf D}}^p(d,\Omega).
\end{array}
$$

Note that $\tilde{\tilde R}_\Omega$ is a potential operator modulo compactness, 
in the sense that, if $u\in{\rm{\sf R}}^p(d,\Omega)$, then 
$u=df+\tilde{\tilde K}_\Omega u$ where $f=\tilde{\tilde R}_\Omega u$.
It is good enough for most purposes, but it can be improved as follows. Define
\begin{equation}
\label{eq:R,K}
R_\Omega:=\tilde{\tilde R}_\Omega +\tilde{\tilde K}_\Omega\tilde{\tilde R}_\Omega
+\tilde{\tilde K}_\Omega\underline{R}\tilde{\tilde K}_\Omega
\quad\mbox{and}\quad
K_\Omega:=\tilde{\tilde K}_\Omega{\mathcal{P}}_{{\rm{\sf N}}(D_{{}^\Vert})}
\tilde{\tilde K}_\Omega,
\end{equation}
where $\underline{R}$ and ${\mathcal{P}}_{{\rm{\sf N}}(D_{{}^\Vert})}$ were defined
in Definition~\ref{def:projections}. On noting that $K_\Omega$ is zero 
on ${\rm{\sf R}}^p(d,\Omega)$, we see that $R_\Omega$ is a true potential operator 
in the sense that, if $u\in{\rm{\sf R}}^p(d,\Omega)$, then $u=df$ where 
$f=R_\Omega u$. It is not as natural in $L^2(\Omega, \Lambda)$ as the potential 
operator $\underline{R}$, but it has the advantage of working for all $p\in(1,\infty)$.
(We remark that a similar improvement could be made to the potential operators in 
strongly Lipschitz domains studied in \cite{CMcI10}.)

Using duality and the Hodge star operator we 
have similar properties for potential operators associated with $\delta$, 
$\underline{d}$ and $\underline{\delta}$. 
We define
\begin{align*}
&\star Q_\Omega u:=(-1)^{\ell-1}R_\Omega(\star\,u),\ 
\star L_\Omega u:=K_\Omega(\star\,u)\quad
\mbox{for an $\ell$-form }u;
\\
&T_\Omega u:=Q_\Omega^*u;
\\
&\star S_\Omega u:=(-1)^{\ell-1}T_\Omega(\star\,u),\ 
\quad\mbox{for an $\ell$-form }u.
\end{align*}

The properties of the operators $R_\Omega$, $S_\Omega$ and $K_\Omega$ 
are summarised in the following proposition. The properties of $T_\Omega$, 
$Q_\Omega$ and $L_\Omega$, can be deduced in a straightforward way.

\begin{proposition}
\label{prop:propR,K}
Suppose $\Omega$ is a very weakly Lipschitz  domain. 
Then the potential operators $R_\Omega$, $S_\Omega$ and $K_\Omega$ 
defined above satisfy for all $p\in (1,\infty)$
\begin{align*}
&R_\Omega: 
L^p(\Omega,\Lambda)\to L^{p^S}(\Omega,\Lambda)\cap{\textsf{D}}^p(d,\Omega),
\quad 
S_\Omega:
L^p(\Omega,\Lambda)\to L^{p^S}(\Omega,\Lambda)\cap
{\textsf{D}}^p(\underline{\delta},\Omega),
\\[4pt]
&K_\Omega:L^p(\Omega,\Lambda)\to L^\infty(\Omega,\Lambda)
\cap{\textsf{D}}^p(d,\Omega), 
\quad
K_\Omega^*:L^p(\Omega,\Lambda)\to L^\infty(\Omega,\Lambda)
\cap{\textsf{D}}^p(\underline{\delta},\Omega),
\\[4pt]
&K_\Omega, K_\Omega^*\mbox{ are compact operators in }L^p(\Omega,\Lambda),
\\[4pt]
&dR_\Omega+R_\Omega d={\rm I}\,-K_\Omega,
\qquad 
\underline{\delta}S_\Omega+S_\Omega\underline{\delta}={\rm I}\,-K_\Omega^*,
\\[4pt]
&dK_\Omega=0,\quad \underline{\delta}K_\Omega^*=0 \quad \mbox{and}\quad
K_\Omega=0 \mbox{ on }{\rm{\sf R}}^p(d,\Omega), \quad K_\Omega^*=0
\mbox{ on }{\rm{\sf R}}^p(\underline{\delta},\Omega),
\\[4pt]
&dR_\Omega u=u\mbox{ if }u\in{\rm{\sf R}}^p(d,\Omega), \quad
\underline{\delta}S_\Omega u=u
\mbox{ if }u\in{\rm{\sf R}}^p(\underline{\delta},\Omega).
\end{align*}
\end{proposition}

As direct consequence we obtain that $dR_\Omega$, 
$\underline{\delta}S_\Omega$, $\underline{d}T_\Omega$, and
$\delta Q_\Omega$ are projections from 
$L^p(\Omega,\Lambda)$ onto the ranges of $d$, $\underline{d}$, $\delta$ or 
$\underline{\delta}$ for all $p\in(1,\infty)$. 

\begin{corollary}
\label{cor:complex-interpolation-scale}
Suppose $\Omega$ is a very weakly Lipschitz domain. Then 
\begin{enumerate}[(i)]
\item
\label{4.2(i)}
for all $p\in(1,\infty)$, the spaces ${\rm{\sf R}}^p(d,\Omega)$, 
${\rm{\sf R}}^p(\underline{d},\Omega)$, ${\rm{\sf R}}^p(\delta,\Omega)$ and 
${\rm{\sf R}}^p(\underline{\delta},\Omega)$ are closed linear subspaces of 
$L^p(\Omega,\Lambda)$; 
\item
\label{4.2(ii)}
for all $p\in(1,\infty)$,the operators $d$, $\underline{d}$, $\delta$ and 
$\underline{\delta}$ are closed (unbounded) operators in $L^p(\Omega,\Lambda)$;
\item
\label{4.2(iii)}
there exist finite dimensional subspaces ${\mathcal{Z}}_d, 
{\mathcal{Z}}_\delta \subset L^\infty(\Omega,\Lambda)$, 
${\mathcal{Z}}_{\underline{d}},{\mathcal{Z}}_{\underline{\delta}}\subset 
\cap_{q<\infty}L^q(\Omega,\Lambda)$ such that 
${\rm{\sf N}}^p(d,\Omega)= {\rm{\sf R}}^p(d,\Omega)\oplus {\mathcal{Z}}_d$, 
${\rm{\sf N}}^p(\delta,\Omega)= 
{\rm{\sf R}}^p(\delta,\Omega)\oplus {\mathcal{Z}}_\delta$, 
${\rm{\sf N}}^p({\underline{d}},\Omega)= 
{\rm{\sf R}}^p({\underline{d}},\Omega)\oplus {\mathcal{Z}}_{\underline{d}}$ 
and ${\rm{\sf N}}^p({\underline{\delta}},\Omega)
= {\rm{\sf R}}^p({\underline{\delta}},\Omega)\oplus {\mathcal{Z}}_{\underline{\delta}}$ 
for all $p\in(1,\infty)$.
\item
\label{4.2(iv)}
The families of spaces $\bigl\{{\rm{\sf R}}^p(d,\Omega),1<p<\infty\bigr\}$, 
$\bigl\{{\rm{\sf R}}^p(\underline{d},\Omega),1<p<\infty\bigr\}$, 
$\bigl\{{\rm{\sf R}}^p(\delta,\Omega),1<p<\infty\bigr\}$ and
$\bigl\{{\rm{\sf R}}^p(\underline{\delta},\Omega),1<p<\infty\bigr\}$
are complex interpolation scales.  So too are the families of nulspaces.
\item
\label{4.2(v)}
When $1<p<q<\infty$, then ${\rm{\sf R}}^q(d,\Omega)= {\rm{\sf R}}^p(d,\Omega)
\cap L^q(\Omega,\Lambda)$, and similarly for the other range spaces.
\item 
\label{4.2(vi)}
When $1<p<q<\infty$, then ${\rm{\sf R}}^q(d,\Omega)$ is dense in 
${\rm{\sf R}}^p(d,\Omega)$, and similarly for the other range spaces.
\end{enumerate}
\end{corollary}

\begin{proof}
\begin{enumerate}[$(i)$]
\item
This follows from the fact  that the ranges are images of bounded projections. 
\item
The cases of $\delta$, $\underline{d}$ and $\underline{\delta}$ are similar to the 
case of $d$.
Let $(u_k)_{k\in{\mathbb{N}}}$ be a sequence in ${\textsf{D}}^p(d,\Omega)$ 
converging to $u$ in $L^p(\Omega,\Lambda)$ such that 
$(du_k)_{k\in{\mathbb{N}}}$ converges to $v$ in $L^p(\Omega,\Lambda)$. 
By $(i)$, $v\in {\rm{\sf R}}^p(d,\Omega)$ (in particular,
$v=dR_\Omega v$) and by Proposition~\ref{prop:propR,K}, we have that
$u=dR_\Omega u+R_\Omega v+K_\Omega u$. Therefore 
$u\in {\textsf{D}}^p(d,\Omega)$ satisfies 
$du=d(dR_\Omega u+R_\Omega v+K_\Omega u)=dR_\Omega v=v$ since
$d^2=0$ and $dK_\Omega=0$. This proves
that $d$ is a closed operator in $L^p(\Omega,\Lambda)$.
\item 
We just consider the case of $d$. Let 
${\mathcal{Z}}^p_d = K_\Omega({\rm{\sf N}}^p(d,\Omega))\subset 
L^\infty(\Omega)$. Then ${\rm{\sf N}}^p(d,\Omega) 
= {\rm{\sf R}}^p(d,\Omega)\oplus {\mathcal{Z}}^p_d$ 
with decomposition $u=dR_\Omega u + K_\Omega u$ for all 
$u\in {\rm{\sf N}}^p(d,\Omega)$. So the spaces in the decomposition are closed, 
and ${\mathcal{Z}}^p_d$ is finite dimensional (on account of the compactness 
of $K_\Omega$). Moreover if $u\in{\rm{\sf N}}^q(d,\Omega)$, then 
${K_\Omega} u = {K_\Omega}^2 u \in K_\Omega({\rm{\sf N}}^p(d,\Omega)) 
={\mathcal{Z}}^p_d$ so that ${\mathcal{Z}}^q_d\subset {\mathcal{Z}}^p_d$, 
and conversely ${\mathcal{Z}}^p_d\subset {\mathcal{Z}}^q_d$, implying that 
the spaces ${\mathcal{Z}}^p_d$ are independent of $p$ and can just be named 
${\mathcal{Z}}_d$.
\item
The spaces $L^p(\Omega, \Lambda)$ interpolate by the complex method, and 
hence so do their images under bounded projections 
(see \cite[Chap.\,1, \S14.3]{LiMa68}; see also \cite[Lemma~2.12]{MM08}).
\item 
If $u \in {\rm{\sf R}}^p(d,\Omega)\cap L^q(\Omega,\Lambda)$, then 
$u=d(R_\Omega u) \in {\rm{\sf R}}^q(d,\Omega)$.  
\item
$L^q(\Omega,\Lambda)$ is dense in $L^p(\Omega,\Lambda)$, and so then is 
$dR_\Omega L^q(\Omega,\Lambda)$ dense in 
$dR_\Omega  L^p(\Omega,\Lambda)$.
\qedhere
\end{enumerate}
\end{proof} 

We have now essentially proved Remark~\ref{rem:Hodge-dec-(H)}. In particular 
the $L^2$ range spaces are all closed, and for the space $L^2(\Omega,\Lambda)$, 
the following decompositions are equally valid:
\begin{align*}
L^2(\Omega,\Lambda)&={{\rm{\sf R}}(d)}\stackrel{\bot}{\oplus}
{\rm{\sf R}}(\underline{\delta})\stackrel{\bot}{\oplus}{\rm{\sf N}}(D_{{}^\Vert})
={\rm{\sf R}}(d)\stackrel{\bot}{\oplus}{\rm{\sf R}}(\underline{\delta}) 
\oplus {\mathcal{Z}}_{\underline{\delta}}
={\rm{\sf R}}(\underline{\delta})\stackrel{\bot}{\oplus}{\rm{\sf R}}(d)
\oplus {\mathcal{Z}}_d
\\
&={\rm{\sf R}}(\delta)\stackrel{\bot}{\oplus}{\rm{\sf R}}(\underline{d})
\stackrel{\bot}{\oplus}{\rm{\sf N}}(D_{{}^\bot})
= {\rm{\sf R}}(\delta)\stackrel{\bot}{\oplus}{\rm{\sf R}}(\underline{d})
\oplus {\mathcal{Z}}_{\underline{d}}
={\rm{\sf R}}(\underline{d})\stackrel{\bot}{\oplus}{\rm{\sf R}}(\delta)
\oplus {\mathcal{Z}}_{\delta}
\end{align*}
So the spaces ${\rm{\sf N}}(D_{{}^\Vert})$, ${\mathcal{Z}}_{\underline{\delta}}$ and 
${\mathcal{Z}}_d$ all have the same finite dimension (as indeed do their 
components of $\ell$ forms in $L^2(\Omega,\Lambda^\ell)$, which can be 
identified with the de Rham cohomology spaces of $\Omega$ with tangential 
(absolute) boundary conditions, and thus have dimensions determined by the 
global topology of $\Omega$), as do the spaces ${\rm{\sf N}}(D_{{}^\bot})$, 
${\mathcal{Z}}_{\underline{d}}$ and ${\mathcal{Z}}_\delta$ (and their components, 
which can be identified with the de Rham cohomology spaces of $\Omega$ with 
normal (relative) boundary conditions).

A further important consequence of the existence of these potentials is the
fact that the above Hodge decompositions in $L^2(\Omega,\Lambda)$ 
extend to $L^p(\Omega,\Lambda)$ for $p$ in an interval around $2$.

\begin{theorem}
\label{thm:Hodge-dec-Lp}
Let $\Omega\subset{\mathbb{R}}^n$ be a very weakly Lipschitz domain.
There exist Hodge exponents $p_H$, $p^H = {p_H}'$  with 
$1\le p_H<2<p^H\le\infty$ such that the Hodge decomposition
\begin{equation} 
\label{eq:HodgeLpVert1}\tag{$H_p$}
L^p(\Omega,\Lambda)={\rm{\sf R}}^p(d,\Omega)\oplus
{\rm{\sf R}}^p(\underline{\delta},\Omega)\oplus \bigl({\rm{\sf N}}^p(d,\Omega) 
\cap{\rm{\sf N}}^p(\underline{\delta},\Omega)\bigr)  
\end{equation}
holds if and only if $p_H<p<p^H$. Moreover, for $p$ in this range, 
$D_{{}^\Vert}=d+\underline{\delta}$ (with ${\rm{\sf D}}^p(D_{{}^\Vert}) 
= {\rm{\sf D}}^p(d)\cap{\rm{\sf D}}^p(\underline{\delta})$) is a closed operator in 
$L^p(\Omega,\Lambda)$, and 
${\rm{\sf N}}^p(d,\Omega) \cap{\rm{\sf N}}^p(\underline{\delta},\Omega) 
= {\rm{\sf N}}^p(D_{{}^\Vert}) = {\rm{\sf N}}(D_{{}^\Vert})$, so that 
\begin{equation} 
\label{eq:HodgeLpVert}
L^p(\Omega,\Lambda)={\rm{\sf R}}^p(d,\Omega)\oplus
{\rm{\sf R}}^p(\underline{\delta},\Omega)\oplus {\rm{\sf N}}(D_{{}^\Vert}) ;
\end{equation}  
and also $D_{{}^\bot}= \delta+\underline{d}$ is a closed operator in 
$L^p(\Omega,\Lambda)$ with Hodge decomposition
\begin{equation} 
\label{eq:HodgeLpbot}
L^p(\Omega,\Lambda)={\rm{\sf R}}^p(\delta,\Omega)\oplus
{\rm{\sf R}}^p(\underline{d},\Omega)\oplus
{\rm{\sf N}}(D_{{}^\bot})\ .
\end{equation}
\end{theorem} 

\begin{proof}
Let $p\in(1,\infty)$. The decomposition \eqref{eq:HodgeLpVert1} holds if and only if
\begin{align} 
L^p(\Omega, \Lambda) &= {\rm{\sf R}}^p(d,\Omega)\oplus 
{\rm{\sf N}}^p(\underline{\delta},\Omega) \qquad\text{and} 
\label{eqn: first}\\
L^p(\Omega, \Lambda) &={\rm{\sf N}}^p(d,\Omega)\oplus 
{\rm{\sf R}}^p(\underline{\delta},\Omega) .
\label{eqn: second}
\end{align}
Now each of these decompositions hold for $p=2$, and all of the families 
interpolate with respect to $p$ by the the complex method, so by the properties 
of interpolation together with \v Sne\u\i berg's Theorem \cite{Sn74} (see also 
\cite[Theorem~2.7]{KaMi98}), \eqref{eqn: first} holds if and only if $p$ belongs 
to some open interval $J=(q_\Omega, r_\Omega)$ containing 2, while 
\eqref{eqn: second} holds if and only if $p$ belongs to another open interval, 
which, by duality, is $J'=({r_\Omega}',{{q_\Omega}}')$. Therefore 
\eqref{eq:HodgeLpVert1} holds if and only if $p\in J\cap J'$, i.e. 
$p_H<p<p^H$, where $p_H = \max\{q_\Omega, {r_\Omega}'\}$ and
$p^H = \min\{r_\Omega, {q_\Omega}'\} ={p_H}'$.

Once we have the Hodge decomposition \eqref{eq:HodgeLpVert1}, it is 
straightforward to verify that $D_{{}^\Vert}=d+\underline{\delta}$ is a closed 
operator in $L^p(\Omega,\Lambda)$ (using the closedness of $d$ and 
$\underline{\delta}$ proved in Corollary~\ref{cor:complex-interpolation-scale}$(ii)$),
and that ${\rm{\sf N}}^p(d,\Omega) \cap{\rm{\sf N}}^p(\underline{\delta},\Omega) 
= {\rm{\sf N}}^p(D_{{}^\Vert})$.

Moreover, following the reasoning above, we have that 
$\dim({\rm{\sf N}}^p(D_{{}^\Vert})) =\dim {\mathcal{Z}}_d$ which is independent of 
$p\in(p_H,p^H)$. Now ${\rm{\sf N}}^q(D_{{}^\Vert})\subset 
{\rm{\sf N}}^p(D_{{}^\Vert})$ 
when $p_H<p<q<p^H$, and these nulspaces all have the same dimension, 
so they are all equal to ${\textsf{N}}(D_{{}^\Vert})$.

The results for $D_{{}^\bot} = \delta + \underline{d}$ are proved in a similar 
way, and have the same Hodge exponents by Hodge duality.
\end{proof}

We record the following facts about the closed operator 
$D_{{}^\Vert}=d+\underline{\delta}$ (with ${\rm{\sf D}}^p(D_{{}^\Vert}) 
= {\rm{\sf D}}^p(d)\cap{\rm{\sf D}}^p(\underline{\delta})$) in 
$L^p(\Omega,\Lambda)$. 

\begin{proposition}
\label{prop:density} 
If $p_H<p<q<p^H$, then ${\mathsf R}^p(D_{{}^\Vert}) 
= {\mathsf R}^p(d)\oplus {\mathsf R}^p(\underline{\delta})$ and 
${\mathsf R}^q(D_{{}^\Vert})= 
{\mathsf R}^q(d)\oplus {\rm{\sf R}}^q(\underline{\delta})$ 
is dense in ${\mathsf R}^p(D_{{}^\Vert})$. Moreover ${\mathsf{G}}^q(D_{{}^\Vert})$ 
is dense in ${\mathsf{G}}^p(D_{{}^\Vert})$. 
\end{proposition}

\begin{proof} 
The density of the ranges follows from 
Corollary~\ref{cor:complex-interpolation-scale}$(vi)$. In proving the density of the 
graphs, we assume that $q\leq p^S$. Otherwise we proceed by induction.
Let us introduce the  potential map 
$Z:L^p(\Omega,\Lambda) \to L^q(\Omega,\Lambda)$ defined by 
$$
Zv= {\mathcal{P}}_{{\mathsf{R}}(d)}S_\Omega 
{\mathcal{P}}_{{\mathsf{R}}(\underline{\delta})}v + 
{\mathcal{P}}_{{\mathsf{R}}(\underline{\delta})}
R_\Omega{\mathcal{P}}_{{\mathsf{R}}(d)}v
$$
where $R_\Omega$ and $S_\Omega$ have the properties stated in 
Proposition~\ref{prop:propR,K}. This is a potential map in the sense that, for 
all $v\in {\mathsf{R}}^p(D_{{}^\Vert})$,
\begin{align*}
D_{{}^\Vert} Zv &= \underline{\delta} {\mathcal{P}}_{{\mathsf{R}}(d)} S_\Omega 
{\mathcal{P}}_{{\mathsf{R}}(\underline{\delta})}v
+d{\mathcal{P}}_{{\mathsf{R}}(\underline{\delta})} R_\Omega 
{\mathcal{P}}_{{\mathsf{R}}(d)}v\\
&=\underline{\delta} S_\Omega {\mathcal{P}}_{{\mathsf{R}}(\underline{\delta})}v
+d R_\Omega {\mathcal{P}}_{{\mathsf{R}}(d)}v 
={\mathcal{P}}_{{\mathsf{R}}(\underline{\delta})}v+ {\mathcal{P}}_{{\mathsf{R}}(d)}v
=v\,.
\end{align*} 
Let $(u, D_{{}^\Vert}u)\in {\mathsf{G}}^p(D_{{}^\Vert})$. By density of the ranges, 
there exists a sequence $(w_k)_{k\in{\mathbb{N}}}$ in 
${\mathsf{R}}^q(D_{{}^\Vert})$ such that $w_k\xrightarrow[k\to\infty]{} D_{{}^\Vert}u$
in $L^p$. Let $u_k = Zw_k + (u-ZD_{{}^\Vert}u)$. Then $(u_k)_{k\in{\mathbb{N}}}$ 
is a sequence in $L^q(\Omega,\Lambda)$ (because $Zw_k \in L^q$ and 
$u-ZD_{{}^\Vert}u\in {\mathsf{N}}(D_{{}^\Vert}) \in L^q$) and 
$D_{{}^\Vert}u_k=w_k$, so $(u_k, D_{{}^\Vert}u_k) = 
(u_k,w_k) \in {\mathsf{G}}^q(D_{{}^\Vert})$.
Also $u_k-u =  Z(w_k-D_{{}^\Vert}u) \xrightarrow[n\to\infty]{} 0$, so that 
$(u_k,D_{{}^\Vert}u_k) \xrightarrow[n\to\infty]{} (u,D_{{}^\Vert}u)$ in 
$L^p\oplus L^p$. We conclude that ${\mathsf{G}}^q(D_{{}^\Vert})$ is dense in 
${\mathsf{G}}^p(D_{{}^\Vert})$.
\end{proof}

\begin{remark}
\label{rem:hodgerange}
If $\Omega\subset{\mathbb{R}}^n$ is smooth, it is known that $p_H=1$ and 
$p^H=\infty$ (see \cite[Theorems~2.4.2 and~2.4.14]{Schw95}). We will see 
in Section~\ref{sec:pH} that if $\Omega\subset{\mathbb{R}}^n$ is a bounded 
strongly Lipschitz domain, then $p_H<\frac{2n}{n+1}$ and $p^H>\frac{2n}{n-1}$.  
\end{remark}

\section{Hodge-Dirac operators on very weakly Lipschitz domains}
\label{sec:Dirac}

On any $\Omega\subset{\mathbb{R}}^n$, the Hodge-Dirac operator 
$D_{{}^\Vert}=d+\underline{\delta}$ with domain 
${\rm{\sf D}}(D_{{}^\Vert})={\rm{\sf D}}(d,\Omega)
\cap {\rm{\sf D}}(\underline{\delta},\Omega)$, as defined
in Definition~\ref{def:HodgeDirac}, is self-adjoint in $L^2(\Omega,\Lambda)$.
Therefore, by Remark~\ref{self-adj-fun-calc}, $D_{{}^\Vert}$ is bisectorial of angle 
$\omega=0$ in $L^2(\Omega,\Lambda)$, and,  
for all $\mu\in\bigl(0,\frac{\pi}{2}\bigr)$, $D_{{}^\Vert}$ admits a bounded 
$S_\mu^\circ$ holomorphic functional calculus in $L^2(\Omega,\Lambda)$. 

Our aim in this section is to extend this result to a range of values of $p$ under 
the condition that $\Omega$ satisfies condition \eqref{eq:Omega}. 

In the case of a strongly Lipschitz domain, it has been proved in 
\cite[Theorem~7.1]{MM09a} that the semigroup generated by the Hodge-Laplacian 
$\Delta_{{}^\Vert}=-{D_{{}^\Vert}}^2$ in $L^2(\Omega,\Lambda)$ extends to an
analytic semigroup in $L^p(\Omega,\Lambda)$ if $p_H<p<p^H$. Moreover,
the Riesz transforms $\frac{d}{\sqrt{-\Delta_{{}^\Vert}}}$ and 
$\frac{\underline{\delta}}{\sqrt{-\Delta_{{}^\Vert}}}$ are bounded in 
$L^p(\Omega,\Lambda)$ for $p_H<p<p^H$ as proved in 
\cite[Theorem~1.1]{HMM11}. 

Recall that the results presented here for $D_{{}^\Vert}$ are equally valid for
$D_{{}^\bot}$ (see Remark~\ref{rem:D||,DT}).

\begin{theorem}
\label{thm:mainThm4.1} 
Suppose $\Omega$ is a very weakly Lipschitz domain, $1<p<\infty$, and 
$D_{{}^\Vert} =d+\underline{\delta}$ is the Hodge-Dirac operator in 
$L^p(\Omega,\Lambda)$ with domain ${\rm{\sf D}}^p(D_{{}^\Vert})
={\rm{\sf D}}^p(d,\Omega)\cap {\rm{\sf D}}^p(\underline{\delta},\Omega)$.
\begin{enumerate}[(i)]
\item 
\label{5.1(i)}
If $p_H<p<p^H$, then the operator $D_{{}^\Vert} =d+\underline{\delta}$ 
is bisectorial of angle $\omega=0$ in $L^p(\Omega,\Lambda)$, and for all 
$\mu\in\bigl(0,\frac{\pi}{2}\bigr)$, $D_{{}^\Vert}$ admits a bounded 
$S_\mu^\circ$ holomorphic functional calculus in $L^p(\Omega,\Lambda)$.
\item
\label{5.1(ii)}
Conversely, if, for some $p \in (1,\infty)$, the operator $D_{{}^\Vert}$ 
 is bisectorial with a bounded holomorphic functional calculus in 
$L^p(\Omega,\Lambda)$, then $p_H<p<p^H$.
\item
\label{5.1(iii)}
Moreover, for all $r\in\bigl(\max\{1,{p_H}_S\},p^H\bigr)$ (recall that 
$p_S:=\frac{np}{n+p}$) and all $\theta\in\bigl(0,\frac{\pi}{2}\bigr)$, 
there exists $C_{r,\theta}>0$ such that
\begin{equation}
\label{eq:resRDVert}
({\rm I}\,+z D_{{}^\Vert})^{-1}:
\begin{cases}
{\rm{\sf R}}^r(d,\Omega)\\
{\rm{\sf R}}^r(\underline{\delta},\Omega)
\end{cases}
\longrightarrow L^r(\Omega,\Lambda)
\qquad\forall z\in{\mathbb{C}}\setminus S_\theta,
\end{equation}
with the estimates
\begin{equation}
\label{eq:est_resRVert}
\sup_{z\in{\mathbb{C}}\setminus S_\theta}\|({\rm I}\,+z D_{{}^\Vert})^{-1}u\|_r
\le C_{r,\theta}\|u\|_r
\quad\forall\, u\in {\rm{\sf R}}^r(d,\Omega)\ \mbox{ and }\ 
\forall\, u\in {\rm{\sf R}}^r(\underline{\delta},\Omega).
\end{equation}
For all $\mu\in\bigl(0,\frac{\pi}{2}\bigr)$, there exists a constant $K_{r,\mu}$ such
that for all $f\in\Psi(S_\mu^\circ)$,
\begin{equation}
\label{eq:fcRDVert}
f(D_{{}^\Vert}):\begin{cases}
{\rm{\sf R}}^r(d,\Omega)\\
{\rm{\sf R}}^r(\underline{\delta},\Omega)
\end{cases}
\longrightarrow L^r(\Omega,\Lambda),
\end{equation}
with the estimates
\begin{equation}
\label{eq:est_fcRVert} 
\|f(D_{{}^\Vert})u\|_r \le K_{r,\mu}\|f\|_{L^\infty(S_\mu^\circ)}\|u\|_r,
\quad\forall\, u\in {\rm{\sf R}}^r(d,\Omega)\ \mbox{ and }\ 
\forall\, u\in {\rm{\sf R}}^r(\underline{\delta},\Omega).
\end{equation}
\end{enumerate}
\end{theorem}

\noindent
The proof of this result is iterative.
In the iteration arguments, we will apply the following two intermediate results. 
The heart of the extrapolation method is deferred to Section \ref{sec:general}. 
  
\begin{proposition}
\label{prop:OD} 
Suppose $\Omega$ is a very weakly Lipschitz domain, that $p_H<q<p^H$, 
and that $D_{{}^\Vert}$  is bisectorial of angle 
$\omega\ge0$ in $L^q(\Omega,\Lambda)$.  Suppose 
$\omega<\mu<\frac{\pi}{2}$ and  $\max\{1,q_S\}<p<q$.
\begin{enumerate}[(i)]
\item
\label{5.2(i)}
The family of resolvents
\begin{equation}
\label{eq:OD1}
\bigl\{({\rm I}\,+zD_{{}^\Vert})^{-1}\,;\,z\in{\mathbb{C}}\setminus S_\mu\bigr\}
\end{equation}
admits off-diagonal bounds $L^q-L^q$ as defined in Definition~\ref{def:OD}.
Moreover the following families of operators
\begin{equation}
\label{OD2}
\bigl\{zd({\rm I}\,+zD_{{}^\Vert})^{-1}\,;\, z\in{\mathbb{C}}\setminus S_\mu\bigr\}
\quad\mbox{and}\quad 
\bigl\{z\underline{\delta}({\rm I}\,+zD_{{}^\Vert})^{-1}\,;\, 
z\in{\mathbb{C}}\setminus S_\mu\bigr\}
\end{equation}
also admit off-diagonal bounds $L^q-L^q$, as (by Remark~\ref{rem:OD}) do  
the families of adjoints,
\begin{equation}
\label{OD3}
\bigl\{z\overline{({\rm I}\,+zD_{{}^\Vert})^{-1}\underline{\delta}}\,;\, 
z\in{\mathbb{C}}\setminus S_\mu\bigr\}
\quad\mbox{and}\quad 
\bigl\{z\overline{({\rm I}\,+zD_{{}^\Vert})^{-1}d} \,;\, 
z\in{\mathbb{C}}\setminus S_\mu\bigr\}.
\end{equation}
\item
\label{5.2(ii)}
Condition \eqref{condA} of Theorem~\ref{thm:abstractThm} holds in each of 
the following cases:
\begin{enumerate}[(a)]
\item
The operators $A=D_{{}^\Vert}$, $B=d$ and the subspace
$X_p={\rm{\sf R}}^p(d,\Omega)$, 
\item
the operators $A=D_{{}^\Vert}$, $B=\underline{\delta}$ and the subspace
$X_p={\rm{\sf R}}^p(\underline{\delta},\Omega)$.
\end{enumerate}
\item
\label{5.2(iii)}
There exist  constants $M_{p,\mu}$ such that 
\begin{align*}
\bigl\|({\rm I}\,+zD_{{}^\Vert})^{-1}u\bigr\|_p\le M_{p,\mu} \|u\|_p,\quad
\forall\,z\in{\mathbb{C}}\setminus S_{\mu},\ &
\forall\, u\in {\rm{\sf R}}^p(d,\Omega)\cap L^q(\Omega,\Lambda) 
= {\rm{\sf R}}^q(d,\Omega)\\
\mbox{and}\quad&\forall\, u\in {\rm{\sf R}}^p(\underline{\delta},\Omega)
\cap L^q(\Omega,\Lambda) = {\rm{\sf R}}^q(\underline{\delta},\Omega).
\end{align*}
\item 
\label{5.2(iv)}
If in addition $p>p_H$, then $D_{{}^\Vert}$  is bisectorial of angle 
$\omega$ in $L^p(\Omega,\Lambda)$.
\end{enumerate}
\end{proposition}

\begin{proof}
\eqref{5.2(i)}
The methods used in this proof are inspired 
by those developed for the proof of \cite[Lemma~2.1]{AHLMcIT02}; 
see also \cite[Proposition~5.1]{AAMcI10}.

We start with the proof of off-diagonal bounds for the families \eqref{eq:OD1}.
Let $\mu\in\bigl(\omega,\frac{\pi}{2}\bigr)$ and $E,F\subset{\mathbb{R}}^n$ 
be Borel sets. Let $z\in {\mathbb{C}}\setminus S_\mu$ and $t:=|z|>0$. 
If ${\rm dist}\,(E,F)=0$, the result is immediate since the resolvent
is bounded in $L^q(\Omega,\Lambda)$ by assumption, so suppose 
that ${\rm dist}\,(E,F)>0$.
Let $M_{q,\mu}:=\sup_{z\in {\mathbb{C}}\setminus S_\mu}
\|({\rm I}\,+zD_{{}^\Vert})^{-1}\|_{{\mathscr{L}}(L^q(\Omega,\Lambda))}$. 
Let $\xi$ be a real-valued function satisfying
\begin{equation}
\label{OD4}
\xi\in{\rm Lip}({\mathbb{R}}^n)\cap L^\infty({\mathbb{R}}^n),\
\xi=1\mbox{ on }E,\ \xi=0\mbox{ on }F\mbox{ and } 
\|\nabla\xi\|_\infty\le\frac{1}{{\rm dist}\,(E,F)} ,
\end{equation}
for example taking $\xi(x) = \min\{\frac{{\rm dist}(x,F)}{{\rm dist}(E,F)},1\}$.
Let $\alpha>0$ (which will be determined later) 
and let $\eta:=e^{\alpha\xi}$. Note that $\nabla\eta=\alpha\eta\nabla\xi$.

For each $u\in L^q(\Omega,\Lambda)$, set  
$v:=({\rm I}\,+zD_{{}^\Vert})^{-1}(1\!{\rm l}_Fu)
=({\rm I}\,+zD_{{}^\Vert})^{-1}(\eta\, 1\!{\rm l}_Fu)
\in{\rm{\sf D}}^q(D_{{}^\Vert},\Omega)$ 
(since $\eta=1$ on $F$), noting that $\eta v \in {sf D}^q(D_{{}^\Vert},\Omega)$ 
(see Remark~\ref{product-rule}).
Hence we have the following commutator identity 
\begin{align*} 
\eta v &=v +\bigl[\eta,({\rm I}\,+zD_{{}^\Vert})^{-1}\bigr](1\!{\rm l}_Fu)
\\
&=v -({\rm I}\,+zD_{{}^\Vert})^{-1}\bigl[\eta,zD_{{}^\Vert}\bigr]
({\rm I}\,+zD_{{}^\Vert})^{-1}(1\!{\rm l}_Fu)
\\
&=v+ z({\rm I}\,+zD_{{}^\Vert})^{-1} 
\bigl(\nabla\eta\wedge v-\nabla\eta\lrcorner\,v\bigl) 
\qquad\qquad\text{by \eqref{eq:product-rule}}
\\
&=v+ z({\rm I}\,+zD_{{}^\Vert})^{-1}\alpha
\bigl(\nabla\xi\wedge(\eta v)-\nabla\xi\lrcorner\,(\eta v)\bigl).  
\end{align*} 
Since $\|v\|_q\le M_{q,\mu}\|u\|_q$, we have the estimate
\begin{equation}
\label{OD5}
\|\eta v\|_q\le M_{q,\mu}\|u\|_q
+2\alpha \,t M_{q,\mu} \frac{1}{{\rm dist}\,(E,F)}\,\|\eta v\|_q.
\end{equation}
We now choose  $\alpha=\frac{{\rm dist}\,(E,F)}{4t M_{q,\mu}}$
and since $\eta=e^\alpha$ on $E$, \eqref{OD5} implies
\begin{equation}
\label{OD5b}
e^\alpha\|1\!{\rm l}_Ev\|_q\le \|\eta v\|_q \leq2M_{q,\mu}\|u\|_q.
\end{equation}
Therefore, we have proved off-diagonal bounds (as in Definition~\ref{def:OD})
for $\bigl\{({\rm I}\,+zD_{{}^\Vert})^{-1},z\in{\mathbb{C}}\setminus S_\mu\bigr\}$
with $C=2M_{q,\mu}$ and $c=\frac{1}{4 M_{q,\mu}}$.

We turn now to the proof of off-diagonal bounds for the first family in
\eqref{OD2}, and use the same notation as above for 
$q, t, \xi, \eta, M_{q,\mu}, u, v$,
first noting that
\begin{equation}
\label{OD6}
\|zD_{{}^\Vert}({\rm I}\,+zD_{{}^\Vert})^{-1}u\|_q \le (1+M_{q,\mu})\|u\|_q, 
\quad\forall z\in{\mathbb{C}}\setminus S_\mu.
\end{equation}
Since $q\in(p_H,p^H)$, the Hodge projection 
$P_{{\rm{\sf R}}^q(d,\Omega)}:L^q(\Omega,\Lambda)\to
{\rm{\sf R}}^q(d,\Omega)$ is bounded on $L^q(\Omega,\Lambda)$;
we denote by $M_q$ its norm. It is straightforward 
that $P_{{\rm{\sf R}}^q(d,\Omega)}D_{{}^\Vert}v=dv$ for all 
$v\in{\rm{\sf D}}^q(D_{{}^\Vert},\Omega)$. 
From \eqref{OD6} follows the estimate
\begin{align*}
\|zd({\rm I}\,+zD_{{}^\Vert})^{-1}u\|_q &\le M_q(1+M_{q,\mu})\|u\|_q \quad\forall
z\in{\mathbb{C}}\setminus S_\mu\ , \quad\text{and therefore}
\\
\|zdv\|_q &\le M_q(1+M_{q,\mu})\|u\|_q \quad\forall
z\in{\mathbb{C}}\setminus S_\mu\ . 
\end{align*}
Further
\begin{align*}
\eta z dv-zd v
&=\eta z dv-z d (\eta v)+z d(\eta v)-zd v
\\
&=-\alpha z\nabla\xi\wedge(\eta v)+ zd(\eta v-v)
\\
&=\alpha z\Bigl(-\nabla\xi\wedge(\eta v)+
zd({\rm I}\,+zD_{{}^\Vert})^{-1}
\bigl(\nabla\xi\wedge (\eta v)-\nabla\xi\lrcorner\,(\eta v)\bigr)\Bigr).
\end{align*}
This gives the estimate 
$$
\|\eta z dv\|_q\le
M_q(1+M_{q,\mu})\|u\|_q
+\tfrac{\alpha\, t}{{\rm dist}\,(E,F)}\,
\bigl(1+2M_q(1+M_{q,\mu})\bigr)\|\eta v\|_q.
$$
Choosing $\alpha=\frac{{\rm dist}\,(E,F)}{4 t M_{q,\mu}}$, and using 
the bound proved in \eqref{OD5b} for $\|\eta v\|_q$, we obtain
$$
\|\eta z dv\|_q\le 
\bigl(1/2+2M_q(1+M_{q,\mu})\bigr)\|u\|_q
$$
and  conclude as before that $\bigl\{zd({\rm I}\,+zD_{{}^\Vert})^{-1},
z\in{\mathbb{C}}\setminus S_\mu\bigr\}$ satisfies off-diagonal bounds
with $C=1/2+2M_q(1+M_{q,\mu})$ and 
$c=\frac{1}{4 M_{q,\mu}}$.

\noindent
The proof of the  off-diagonal bound for the other family  in \eqref{OD2}
follows the  same lines.

\noindent 
\eqref{5.2(ii)}
The proofs of points 1 and 2 are similar and rely on the properties 
of the potentials described in Section~\ref{sec:potentials}. We present  
the proof of point~1, so suppose $A=D_{{}^\Vert}$, $B=d$ and 
$X_p={\rm{\sf R}}^p(d,\Omega)$.

For the family $Q^t_k$ required in \eqref{condA} of Theorem~\ref{thm:abstractThm}, 
proceed as follows.
\begin{itemize}
\item
Suppose $0<t\leq {\rm diam}\,\Omega$.
\item 
Cover $\Omega$: Let $\underline Q^t_k$  ($k\in J$) be the cubes in 
${\mathbb{R}}^n$ with side-length $t$ and corners at points in $t{\mathbb{Z}}^n$, 
which intersect $\Omega$. Let $Q^t_k = 4\underline Q^t_k \cap\Omega$. 
Then $\Omega = \cup Q^t_k$. 
\item 
There exist functions $\underline\eta_k\in 
{\mathscr{C}}^1_c(4\underline Q^t_k, [0,1])$ with 
$\|\nabla \underline\eta_k\|_\infty \le 1/t$ and $\sum{\underline\eta_k}^2=1$ 
on $\Omega$. Then $\eta_k:=\underline\eta_k|_\Omega$ is a Lipschitz 
function on $\Omega$ with values in $[0,1]$, ${\rm sppt}_\Omega(\eta_k)
\subset Q^t_k$, $\|\nabla \eta_k\|_\infty \le 1/t$ and $\sum{\eta_k}^2=1$ on 
$\Omega$. 
\item 
$d(\eta_k f) - \eta_k df = (\nabla \eta_k)\wedge f$. 
\end{itemize}
For $u\in {\rm{\sf R}}^p(d,\Omega)\cap L^q(\Omega,\Lambda)$, 
$u=dR_\Omega u$ (where $R_\Omega$ is the potential map defined in 
Section~\ref{sec:potentials}) and we define
$$
w_k=\eta_kR_\Omega(\eta_k u)\quad\mbox{and}\quad
v_k=\eta_kR_\Omega(t\nabla\eta_k\wedge u)
-t\nabla\eta_k\wedge R_\Omega(\eta_ku)+t\eta_kK_\Omega(\eta_ku),
$$
where $K_\Omega$ is defined in Section~\ref{sec:potentials}. It is clear that  
${\rm sppt}_\Omega w_k, {\rm sppt}_\Omega v_k\subset Q^t_k$.
Thanks to the relations listed in Proposition~\ref{prop:propR,K}, it is 
immediate that
$$
{\eta_k}^2 u=\eta_k(dR_\Omega+R_\Omega d +K_\Omega)\eta_ku 
=dw_k+\tfrac{1}{t}v_k
$$
and so
$$
u=\sum_k {\eta_k}^2 u =\sum_k\Bigl(dw_k+\tfrac{1}{t}v_k\Bigr).
$$ 
It remains to prove estimates on $w_k$ and $v_k$. They come from
the mapping properties of $R_\Omega$ and $K_\Omega$. Denote by
$r\in(1,\infty)$ the real number satisfying $\frac{1}{q}=\frac{1}{p^S}+\frac{1}{r}$.
In other words, $r$ satisfies $\frac{1}{r}=\frac{1}{n}-\bigl(\frac{1}{p}-\frac{1}{q}\bigr)$.
We have that 
$$
\|w_k\|_q\lesssim \|\eta_k\|_r\|R_\Omega (\eta_ku)\|_{p^S}
\lesssim |Q_k^t|^{\frac{1}{r}}\|\eta_k u\|_p
\lesssim t^{1-n(\frac{1}{p}-\frac{1}{q})} \|1\!{\rm l}_{Q_k^t}u\|_p
$$
and similarly
\begin{align*}
\|v_k\|_q\lesssim & \|\eta_k\|_r\|R_\Omega (t\nabla\eta_k\wedge u)\|_{p^S}
+\|t\nabla\eta_k\|_r\|R_\Omega(\eta_ku)\|_{p^S}+
t\|\eta_k\|_q\|K_\Omega(\eta_ku)\|_\infty
\\[4pt]
\lesssim&
t^{1-n(\frac{1}{p}-\frac{1}{q})} \bigl(1+t^{\frac{n}{p}}\bigr)
\|1\!{\rm l}_{Q_k^t}u\|_p
\le
t^{1-n(\frac{1}{p}-\frac{1}{q})} \bigl(1+{\rm diam}(\Omega)^{\frac{n}{p}}\bigr)
\|1\!{\rm l}_{Q_k^t}u\|_p,
\end{align*}
and thus the condition \eqref{condA} of Theorem~\ref{thm:abstractThm} is satisfied.

\noindent
\eqref{5.2(iii)}
This is now a consequence of Theorem~\ref{thm:abstractThm}. By density
of ${\rm{\sf R}}^q(d,\Omega)$ in ${\rm{\sf R}}^p(d,\Omega)$ and of
${\rm{\sf R}}^q(\underline{\delta},\Omega)$ in 
${\rm{\sf R}}^p(\underline{\delta},\Omega)$
(see Corollary~\ref{cor:complex-interpolation-scale} \eqref{4.2(vi)}),
the estimate in \eqref{5.2(iii)} holds for all $u\in {\rm{\sf R}}^p(d,\Omega)$ and 
for all $u\in {\rm{\sf R}}^p(\underline{\delta},\Omega)$.

\noindent 
\eqref{5.2(iv)}
It is a consequence of \eqref{5.2(iii)} and the Hodge decomposition of 
$L^p(\Omega,\Lambda)$, that there exist  constants $M_{p,\mu}$ such that 
$$
\bigl\|({\rm I}\,+zD_{{}^\Vert})^{-1}u\bigr\|_p\le M_{p,\mu} \|u\|_p,\quad
\forall\,z\in{\mathbb{C}}\setminus S_{\mu},\ 
\forall\, u\in  L^q(\Omega,\Lambda)\,.
$$
By the density of ${\mathsf{G}}^q(D_{{}^\Vert})$ in ${\mathsf{G}}^p(D_{{}^\Vert})$ 
(Proposition~\ref{prop:density}) it then follows that the $L^p$ operator 
$({\rm I}\,+zD_{{}^\Vert})$ is invertible in $L^p(\Omega,\Lambda)$, with
\begin{equation*}
\bigl\|({\rm I}\,+zD_{{}^\Vert})^{-1}u\bigr\|_p\le M_{p,\mu} \|u\|_p,\quad
\forall\,z\in{\mathbb{C}}\setminus S_{\mu},\ 
\forall\, u\in  L^p(\Omega,\Lambda)\,.
\qedhere
\end{equation*}
\end{proof}

\begin{proposition}
\label{prop:hyp-abstractThm}
Suppose that in addition to the hypotheses of Proposition~\ref{prop:OD}, that 
$D_{{}^\Vert}$ has a bounded holomorphic functional calculus in 
$L^q(\Omega,\Lambda)$ with $p_H<q<p^H$ and $\max\{1,q_S\}<p<q$.

Then condition \eqref{condB} of Theorem~\ref{thm:abstractThm2} holds in each 
of the following cases:
\begin{enumerate}
\item
The operators $A=D_{{}^\Vert}$, $B=d$ and the subspace
$X_p={\rm{\sf R}}^p(d,\Omega)$, 
\item
the operators $A=D_{{}^\Vert}$, $B=\underline{\delta}$ and the subspace
$X_p={\rm{\sf R}}^p(\underline{\delta},\Omega)$.
\end{enumerate}  
Consequently for each $r\in (p,q)$ there exist  constants $\kappa_{r,\mu}$ such that 
\begin{align}
\label{eq:functcalcR}
\bigl\|f(D_{{}^\Vert})u\bigr\|_r\le \kappa_{r,\mu} \|f\|_\infty \|u\|_r,\quad
\forall\,z\in{\mathbb{C}}\setminus S_\mu,\ 
&\forall\, u\in {\rm{\sf R}}^r(d,\Omega)\quad\text{and}\\
&\forall\, u\in {\rm{\sf R}}^r(\underline{\delta},\Omega)
\nonumber
\end{align}
for all $f\in\Psi(S_\mu^\circ)$.
\end{proposition}

\begin{proof}
Our aim is to prove that the condition \eqref{condB} of 
Theorem~\ref{thm:abstractThm2} is satisfied. Let $\alpha>0$, 
$u\in X_p\cap L^q(\Omega,\Lambda)$ and let
$$
F:=\bigl\{x\in{\mathbb{R}}^n ;
\bigl({\mathcal{M}}(|\tilde u|^p)(x)\bigr)^{\frac{1}{p}}\le\alpha\bigr\},\quad
E_\alpha:={\mathbb{R}}^n\setminus F,
$$
where ${\mathcal{M}}$ denotes the uncentered Hardy-Littlewood maximal 
operator on ${\mathbb{R}}^n$, i.e.,
$$
{\mathcal{M}}(f)(x):=\sup_{Q\ni x}\fint_Q f(y)\,{\rm d}y, \quad x\in{\mathbb{R}}^n,
\quad f\in L^1_{\rm loc}({\mathbb{R}}^n),
$$
where the $\sup$ is taken over all cubes $Q\subset{\mathbb{R}}^n$ 
containing $x$ and $\tilde{u}$ denotes the extension by zero to ${\mathbb{R}}^n$
of $u$. Let $Q_k=Q(x_k,t_k)$, $k\in{\mathbb{N}}$ be the family 
of cubes relative to $F$ given by \cite[Chap. I, \S3, Theorem~3]{St70} and denote 
by $2^jQ_k$ the dilated cube $Q(x_k,2^jt_k)$. Since 
$2 Q_k\cap F\neq\emptyset$, we have that
$$
\int_{Q_k\cap\Omega}|u|^p\,{\rm d}x=\int_{Q_k}|\tilde u|^p\,{\rm d}x
\le |2 Q_k|\fint_{\lambda Q_k}|\tilde u|^p\,{\rm d}x
\lesssim \alpha^p |Q_k|.
$$
Moreover, by the finite overlapping property of the family of $Q_k$'s and
the properties of the maximal operator (see, e.g., 
\cite[Chap. I, \S1, Theorem~1]{St70}), we have that
$$
\sum_k|Q_k|\lesssim\Bigl|\bigcup_kQ_k\Bigr|=|E_\alpha|=
\Bigl|\bigl\{x\in{\mathbb{R}}^n;{\mathcal{M}}(|\tilde u|^p)(x)>\alpha^p\bigr\}\Bigr|
\lesssim \frac{1}{\alpha^p}\,\||\tilde u|^p\|_{L^1({\mathbb{R}}^n)}
=\frac{1}{\alpha^p}\,\|u\|_p^p.
$$
Next, for each $k\in{\mathbb{N}}$, let $\eta_k\in{\mathscr{C}}_c^\infty(Q_k,[0,1])$ 
be such that $\sum_k\eta_k^2=1\!{\rm l}_{E_\alpha}$ and 
$\|\nabla\eta_k\|_\infty\lesssim\frac{1}{t_k}$. We define $g$ by
$g:=1\!{\rm l}_{\Omega\setminus E_\alpha}u$. It is clear that
$\|g\|_p\le \|u\|_p$ and by Lebesgue differentiation Theorem, we have that
$$
|g(x)|\le \alpha\quad\mbox{for almost all }x\in\Omega.
$$
We define next, for the relevant $k\in{\mathbb{N}}$, i.e., those $k\in{\mathbb{N}}$ 
such that $Q_k\cap\Omega\neq\emptyset$ and $t_k\le{\rm diam}\,\Omega$,
$$
w_k:=\eta_kR_\Omega(\eta_k u)\quad\mbox{and}\quad
v_k:=\eta_kR_\Omega(t_k\nabla\eta_k\wedge u)-t_k\nabla\eta_k\wedge
R_\Omega(\eta_ku)-t_k\eta_kK_\Omega(\eta_ku).
$$
Since $\eta_k$ is smooth and $R_\Omega(\eta_k u)\in {\rm{\sf R}}^p(d,\Omega)$,
it follows that $w_k\in {\rm{\sf R}}^p(d,\Omega)$. We have that 
$\eta_k^2 u=dw_k+\frac{1}{t_k}v_k$, and therefore
$$
u=g+\sum_k\Bigl(dw_k+\textstyle{\frac{1}{t_k}}v_k\Bigr).
$$
Moreover, $w_k$ and $v_k$, $k\in{\mathbb{N}}$, satisfy the estimate
$$
\|w_k\|_q, \|v_k\|_q\lesssim t_k^{1-n(\frac{1}{p}-\frac{1}{q})}
\|1\!{\rm l}_{Q_k\cap\Omega}u\|_p,
$$
which proves that in our case, the conditions of \eqref{condB} of 
Theorem~\ref{thm:abstractThm2} are satisfied. Therefore, applying
the result of Theorem~\ref{thm:abstractThm2}, we obtain the following
weak $L^p$-estimate:
\begin{align*}
\bigl\|f(D_{{}^\Vert})u\bigr\|_{p,w}\le K_{p,\mu} \|f\|_\infty \|u\|_p,\quad
\forall\,z\in{\mathbb{C}}\setminus S_\mu,\ 
&\forall\, u\in {\rm{\sf R}}^q(d,\Omega)={\rm{\sf R}}^p(d,\Omega)\cap
L^q(\Omega,\Lambda)\quad\text{and}\\
&\forall\, u\in {\rm{\sf R}}^q(\underline{\delta},\Omega)=
{\rm{\sf R}}^p(\underline{\delta},\Omega)\cap L^q(\Omega,\Lambda).
\end{align*}
By interpolation between this last result and the fact that $D_{{}^\Vert}$
has a bounded holomorphic functional calculus in $L^q(\Omega,\Lambda)$,
and using the density of ${\rm{\sf R}}^q(d,\Omega)$ in ${\rm{\sf R}}^r(d,\Omega)$
and of ${\rm{\sf R}}^q(\underline{\delta},\Omega)$ in 
${\rm{\sf R}}^r(\underline{\delta},\Omega)$ for all $p<r<q$
(see Corollary~\ref{cor:complex-interpolation-scale} \eqref{4.2(vi)}), 
we obtain \eqref{eq:functcalcR}.
\end{proof}

\begin{proof}[Proof of Theorem~\ref{thm:mainThm4.1}]
We are now in position to prove our main theorem.

\noindent
The assertion \eqref{5.1(iii)} is proved by iteration: we start with $q=2$ and
apply Proposition~\ref{prop:OD} \eqref{5.2(iii)}
and Proposition~\ref{prop:hyp-abstractThm} to obtain \eqref{5.1(iii)}
for all $r\in\bigl(\max\{1,2_S\},2\bigr]$. We iterate the procedure $a$ times where 
$a$ is the smallest integer defined by $\frac{2n}{n+2a}<(p_H)_S$ (we can take
$a=1+E\bigl(\frac{n}{2}\bigr)$ were $E(s)$ denotes the integer part of a real $s$) 
and we obtain \eqref{5.1(iii)} for all $r\in \bigl(\max\{1,{p_H}_S\},2\bigr]$. The 
range $[2,p^H)$ is obtained by taking adjoints in the interval $(p_H,2]$.

\noindent
\eqref{5.1(iii)}$\implies$\eqref{5.1(i)}:
For $p$ in the range where \eqref{eq:HodgeLpVert1} holds, it is immediate that
for all $u\in L^p(\Omega,\Lambda)$, and all $z\in{\mathbb{C}}\setminus S_\theta$,
$\theta \in (0,\frac{\pi}{2})$,
$$
({\rm I}\,+zD_{{}^\Vert})^{-1}u=({\rm I}\,+zD_{{}^\Vert})^{-1}
\bigl({\mathcal{P}}_{{\mathsf{R}}^p(d)}u\bigr) +
({\rm I}\,+zD_{{}^\Vert})^{-1}
\bigl({\mathcal{P}}_{{\mathsf{R}}^p(\underline{\delta})}u\bigr)
+{\mathcal{P}}_{{\mathsf{N}}^p(D_{{}^\Vert})}u,
$$
and therefore, by \eqref{eq:est_resRVert},
$$
\|({\rm I}\,+zD_{{}^\Vert})^{-1}u\|_p\le C{p,\theta}
\bigl(\|{\mathcal{P}}_{{\mathsf{R}}^p(d)}u\|_p+
\|{\mathcal{P}}_{{\mathsf{R}}^p(\underline{\delta})}u\|_p\bigr)
+\|{\mathcal{P}}_{{\mathsf{N}}^p(D_{{}^\Vert})}u\|_p\le (C_{p,\theta}+1)\|u\|_p.
$$
Similarly, for all $f\in \Psi(S_\mu^\circ)$, $\mu\in(0,\frac{\pi}{2})$,
$$
f(D_{{}^\Vert})u=f(D_{{}^\Vert}) \bigl({\mathcal{P}}_{{\mathsf{R}}^p(d)}u\bigr) +
f(D_{{}^\Vert})\bigl({\mathcal{P}}_{{\mathsf{R}}^p(\underline{\delta})}u\bigr),
$$
which gives the estimate
$$
\|f(D_{{}^\Vert})u\|_p
\le K_{p,\mu}\bigl(\|{\mathcal{P}}_{{\mathsf{R}}^p(d)}u\|_p+
\|{\mathcal{P}}_{{\mathsf{R}}^p(\underline{\delta})}u\|_p\bigr)
\le K_{p,\mu}\|u\|_p
$$
thanks to \eqref{eq:fcRDVert}.

\noindent
\eqref{5.1(ii)}: Assume that $p$ is such that
$D_{{}^\Vert}$ admits a bounded $S_\mu^\circ$ holomorphic functional calculus
in $L^p(\Omega,\Lambda)$. The fact that $D_{{}^\Vert}$ is
bisectorial in $L^p(\Omega,\Lambda)$ implies that 
$$
L^p(\Omega,\Lambda)=\overline{{\rm{\sf R}}^p(D_{{}^\Vert})}\oplus
{\rm{\sf N}}^p(D_{{}^\Vert}),
$$
the projections on each subspace being bounded.
See, e.g., \cite[Theorem~3.8]{CDMcIY96}. 
Then the restriction of $D_{{}^\Vert}$ in 
$Y_p:=\overline{{\rm{\sf R}}^p(D_{{}^\Vert})}$ with domain 
${\rm{\sf D}}^p(D_{{}^\Vert})\cap Y_p$ is densely defined, one-to-one and admits 
a bounded $S_\mu^\circ$ holomorphic functional calculus in $Y_p$.
Following the idea of \cite[\S5.3]{AMcIR08}, let ${\rm sgn}$ be the (bounded)
holomorphic function in $S_\mu^\circ$ defined by ${\rm sgn}(z)=\frac{z}{\sqrt{z^2}}$
where $\sqrt{\cdot}$ is the holomorphic continuation of 
$(0,+\infty)\ni x\mapsto\sqrt{x}$ to ${\mathbb{C}}\setminus (-\infty,0]$. Then
we have that ${\rm sgn}^2(D_{{}^\Vert})u={\rm sgn}({\rm sgn}(D_{{}^\Vert})u)=u$ 
for all $u\in Y_p$. Now, \eqref{eq:HodgeLpVert1} is a consequence
of $\|D_{{}^\Vert}u\|_p\approx \|du\|_p+\|\underline{\delta}u\|_p\approx
\bigl\|\sqrt{D_{{}^\Vert}^2}\,u\bigr\|_p$. Indeed, assuming these equivalences 
hold, for all $u\in Y_p$, 
\begin{align*}
&u=dv+\underline{\delta}w,\quad\mbox{ where }\quad
v=\frac{\underline{\delta}}{D_{{}^\Vert}^2}\,u\quad\mbox{and}\quad
w=\frac{d}{D_{{}^\Vert}}\,u,
\\
&v\in {\rm{\sf D}}^p(d,\Omega),\  \|dv\|_p\lesssim \|u\|_p \quad\mbox{and}\quad
w\in {\rm{\sf D}}^p(\underline{\delta},\Omega), \ 
\|\underline{\delta}w\|_p\lesssim \|u\|_p.
\end{align*}
The equivalence $\|D_{{}^\Vert}u\|_p\approx \bigl\|\sqrt{D_{{}^\Vert}^2}\,u\bigr\|_p$
comes from the boundedness of the holomorphic functional calculus for
$D_{{}^\Vert}$ in $Y_p$. To prove $\|D_{{}^\Vert}u\|_p\approx 
\|du\|_p+\|\underline{\delta}u\|_p$, it is sufficient to show that
$\|du\|_p\lesssim\|D_{{}^\Vert}u\|_p$ for all $u\in {\rm{\sf D}}^p(D_{{}^\Vert})$.
Write $\displaystyle{u=\sum_{k=0}^n u^k}$ where 
$u^k\in L^p(\Omega,\Lambda^k)$. Then
\begin{align*}
&\|du\|_p\approx\sum_{k=0}^n\|(du)^k\|_p=\sum_{\ell=0}^n\|d(u^\ell)\|_p
\le \sum_{\ell=0}^n\|D_{{}^\Vert}(u^\ell)\|_p\approx
\sum_{\ell=0}^n\bigl\|\sqrt{D_{{}^\Vert}^2}(u^\ell)\bigr\|_p
\\
=&\sum_{\ell=0}^n\bigl\|(\sqrt{D_{{}^\Vert}^2}u)^\ell\bigr\|_p
\approx \bigl\|\sqrt{D_{{}^\Vert}^2}u\bigr\|_p\approx \|D_{{}^\Vert}u\|_p.
\end{align*}
The bound $\displaystyle{\sum_{\ell=0}^n\|d(u^\ell)\|_p\le 
\sum_{\ell=0}^n\|D_{{}^\Vert}(u^\ell)\|_p}$
holds because $d(u^\ell)\in L^p(\Omega,\Lambda^{\ell+1})$ and
$\underline{\delta}(u^\ell)\in L^p(\Omega,\Lambda^{\ell-1})$.
This proves then that \eqref{eq:HodgeLpVert1} holds if $p$ is as in
\eqref{5.1(ii)}.
\end{proof}

\section{Perturbed Hodge-Dirac operators on strongly Lipschitz domains}
\label{sec:pertDirac}

Let $\Omega\subset{\mathbb{R}}^n$ be a bounded strongly Lipschitz domain.
Let $\phi:{\mathbb{R}}^n\to{\mathbb{R}}^n$ a bilipschitz map as in 
Proposition~\ref{prop:lip-sim-smooth} for which $\Omega'=\phi^{-1}(\Omega)$ 
is a smooth domain. The following result is the perturbed version of 
Theorem~\ref{thm:mainThm4.1} in the case of bounded strongly Lipschitz domains.
 
\begin{theorem}
\label{thm:pertDirac}
Let $B\in L^\infty(\Omega,{\mathscr{L}}(\Lambda))$ such that 
$\Re e B \geq \kappa I$ ($\kappa>0$) and $B(x)$ 
is invertible for almost all $x\in\Omega$. We assume moreover that
$B^{-1}:\Omega\to{\mathscr{L}}(\Lambda)$ defined by $B^{-1}(x):=(B(x))^{-1}$
belongs to $L^\infty(\Omega,{\mathscr{L}}(\Lambda))$. Let 
$D_{{}^\Vert,B}$ be the (unbounded) operator defined on $L^2(\Omega,\Lambda)$
by 
$$
D_{{}^\Vert,B}=d+\underline{\delta}_B=d+B^{-1}\underline{\delta}B \quad
{\rm{\sf D}}(D_{{}^\Vert,B})={\rm{\sf D}}(d)\cap 
{\rm{\sf D}}(\underline{\delta}B).
$$
Then there exist $\omega_B\in\bigl[0,\frac{\pi}{2}\bigr)$ and 
$\varepsilon_B,\tilde\varepsilon_B>0$ 
such that for all $\theta\in\bigl(\omega_B,\frac{\pi}{2}\bigr)$ and all
$p\in\bigl(\max\{1,(2-\tilde\varepsilon_B)_S\}, 2+\varepsilon_B\bigr)$,
there exists $C_{p,\theta}>0$ such that
\begin{equation}
\label{eq:resRDVertpert}
({\rm I}\,+z D_{{}^\Vert,B})^{-1}:
\begin{cases}
{\rm{\sf R}}^p(d,\Omega)\\
{\rm{\sf R}}^p(\underline{\delta}_B,\Omega)
\end{cases}
\longrightarrow L^p(\Omega,\Lambda),
\qquad\forall z\in{\mathbb{C}}\setminus S_\theta,
\end{equation}
with the estimates
\begin{equation}
\label{eq:est_resRVertpert}
\sup_{z\in{\mathbb{C}}\setminus S_\theta}\|({\rm I}\,+z D_{{}^\Vert,B})^{-1}u\|_p
\le C_{p,\theta}\|u\|_p,
\quad\forall u\in {\rm{\sf R}}^p(d,\Omega)\ \mbox{ and }\ 
\forall u\in {\rm{\sf R}}^p(\underline{\delta}_B,\Omega).
\end{equation}
For all $\mu\in\bigl(0,\frac{\pi}{2}\bigr)$, there exists a constant $K_{p,\mu}$ such
that for all $f\in\Psi(S_\mu^\circ)$,
\begin{equation}
\label{eq:fcRDVertpert}
f(D_{{}^\Vert,B}):\begin{cases}
{\rm{\sf R}}^p(d,\Omega)\\
{\rm{\sf R}}^p(\underline{\delta}_B,\Omega)
\end{cases}
\longrightarrow L^p(\Omega,\Lambda),
\end{equation}
with the estimates
\begin{equation}
\label{eq:est_fcRVertpert} 
\|f(D_{{}^\Vert,B})u\|_p \le K_{p,\mu}\|f\|_{L^\infty(S_\mu^\circ)}\|u\|_p,
\quad\forall u\in {\rm{\sf R}}^p(d,\Omega)\ \mbox{ and }\ 
\forall u\in {\rm{\sf R}}^p(\underline{\delta}_B,\Omega).
\end{equation}
\end{theorem}

\begin{proof}
The proof follows the lines of the proof of Theorem~\ref{thm:mainThm4.1}.
Let $\tilde\varepsilon_B,\varepsilon_B>0$ such that the Hodge decomposition
$$
L^p(\Omega,\Lambda)={\rm{\sf R}}^p(d,\Omega)\oplus
{\rm{\sf R}}^p(\underline{\delta}_B,\Omega)\oplus
{\rm{\sf N}}^p(d+\underline{\delta}_B,\Omega)
$$
holds for all $p\in \bigl(2-\tilde\varepsilon_B,2+\varepsilon_B\bigr)$.
We first note that the result is true for $p=2$. To prove that 
$D_{{}^\Vert,B}$ admits a bounded holomorphic functional calculus in 
$L^2(\Omega,\Lambda)$, we use the characterization of 
\cite[Theorem~2]{AKMcI06} after transformation of the problem in the 
smooth domain $\Omega'$
from Proposition~\ref{prop:lip-sim-smooth}: $\Omega=\phi(\Omega')$
where $\phi:{\mathbb{R}}^n\to{\mathbb{R}}^n$ is a bilipschitz map.
The triplet $(d,\tilde{B}^{-1},\tilde{B})$ with 
$\tilde{B}=(\tilde{\phi_*})^{-1}B(\phi^*)^{-1}$ satisfies the conditions
$(H1)-(H8)$ of \cite{AKMcI06} (the condition $(H8)$ is satisfied thanks to
the embedding of ${\rm{\sf D}}(d+\underline{\delta},\Omega')$ into 
$H^1(\Omega',\Lambda)$ since $\Omega'$ is smooth: see 
Remark~\ref{rem:smooth-or-convex}). We conclude then that the operator 
$d+\underline{\delta}_{\tilde{B}}$ admits a bounded holomorphic
functional calculus in $L^2(\Omega',\Lambda)$. Therefore, 
$d+\underline{\delta}_B$ admits a bounded holomorphic functional
calculus in $L^2(\Omega,\Lambda)$.
Next, instead of potentials $R_\Omega$ and $S_\Omega$, we use 
$R_\Omega$ and $B^{-1}S_\Omega B$ which have the same mapping 
properties as $R_\Omega$ and $S_\Omega$ listed in 
Proposition~\ref{prop:propR,K}. This gives the result in the range 
$\bigl(\max\{1,(2-\tilde\varepsilon_B)_S\}, 2\bigr]$. To 
obtain the range $[2, 2+\varepsilon_B)$, we proceed by duality,
using $\underline{\delta}+B^*d(B^*)^{-1}$ the adjoint of $D_{{}^\Vert,B}$ and 
the potential maps
$S_\Omega$ and $B^*R_\Omega (B^*)^{-1}$ instead of $R_\Omega$
and $S_\Omega$.
\end{proof}

\section{Estimates of the Hodge exponents on strongly Lipschitz domains}
\label{sec:pH}

In this section, we focus on the case of bounded strongly Lipschitz domains.
We start with a result which gives good integrability properties of solutions
of $D_{{}^\Vert}u=f$ on $\Omega$ when $\Omega\subset{\mathbb{R}}^n$
is a bounded strongly Lipschitz domain. We recall that, according to
\cite[Theorem~11.2]{MMT01}, there exists $0<\varepsilon'\le1$ depending on
the geometry of $\Omega$ such that for all $r\in(2-\varepsilon',2+\varepsilon')$,
there is a constant $C>0$ with
\begin{equation}
\label{eq:solvability}
\|u\|_{B^{r,r^\sharp}_{1/r}(\Omega,\Lambda)}
\le C\,\bigl(\|u\|_r+\|du\|_r+\|\delta u\|_r+
\|\nu\lrcorner\, u\|_{L^r(\partial\Omega,\Lambda)}\bigr),
\end{equation}
where $r^\sharp:=\max\{2,r\}$. This estimate is also true if we replace 
$\|\nu\lrcorner\, u\|_{L^r(\partial\Omega,\Lambda)}$ in \eqref{eq:solvability} by 
$\|\nu\wedge u\|_{L^r(\partial\Omega,\Lambda)}$.
Applying \cite[Corollary~2, page\,36]{RS96}, 
we can show that the embedding
\begin{equation}
\label{eq:besov-in-L}
B^{r,r^\sharp}_{1/r}\hookrightarrow L^{r^*}
\end{equation}
holds as long as $r^\sharp\le r^*$. In particular \eqref{eq:besov-in-L} is true for 
all $r\ge\frac{2(n-1)}{n}=2-\frac{2}{n}$. Combining \eqref{eq:solvability} and
\eqref{eq:besov-in-L}, we obtain
\begin{equation}
\label{eq:solvability2}
\|u\|_{r^*}
\le C\,\bigl(\|u\|_r+\|du\|_r+\|\delta u\|_r+
\|\nu\lrcorner\, u\|_{L^r(\partial\Omega,\Lambda)}\bigr)
\end{equation}
and
\begin{equation}
\label{eq:solvability3}
\|u\|_{r^*}
\le C\,\bigl(\|u\|_r+\|du\|_r+\|\delta u\|_r+
\|\nu\wedge u\|_{L^r(\partial\Omega,\Lambda)}\bigr)
\end{equation}
for all $r\in\bigl(2-\min\bigl\{\varepsilon',\frac{2}{n}\bigr\},2+\varepsilon'\bigr)$.
By Theorem~\ref{thm:Hodge-dec-Lp}, we know that there exits $\varepsilon>0$
such that
the Hodge decompositions \eqref{eq:HodgeLpVert} and \eqref{eq:HodgeLpbot}
hold for all $p\in\bigl((2+\varepsilon)',2+\varepsilon\bigr)$. Let 
$\alpha:=\min\bigl\{\varepsilon,\varepsilon',\frac{2}{n-2}\bigr\}>0$. Remark that
this particular choice of $\alpha$ ensures that \eqref{eq:solvability2},
\eqref{eq:solvability3}, \eqref{eq:HodgeLpVert} and \eqref{eq:HodgeLpbot}
hold in the interval $\bigl((2+\alpha)',2+\alpha\bigr)$.

We have the following result.

\begin{theorem}
\label{thm:pHlip}
Let $\Omega\subset{\mathbb{R}}^n$ be a bounded strongly Lipschitz domain.
Then we can estimate the Hodge exponents associated to the Hodge
decompositions \eqref{eq:HodgeLpVert} and \eqref{eq:HodgeLpbot} as follows
$$
p_H\le \bigl((2+\alpha)^*\bigr)'=\frac{(2+\alpha)n}{n(1+\alpha)+1}
<\frac{2n}{n+1}=(2^*)'<\frac{2n}{n-1}=2^*
<\frac{(2+\alpha)n}{n-1}=(2+\alpha)^*\le p^H.
$$
In particular, in dimension $n=2$, we have that $p_H<\frac{4}{3}<4<p^H$
and in dimension $n=3$, we have that $p_H<\frac{3}{2}$ and
$p^H>3$.
\end{theorem}

Before proving this theorem, we first give some properties of the null
space of the operator $D_{{}^\Vert}$ or $D_{{}^\bot}$.

\begin{lemma}
\label{lem:nullspace}
Let $r\in\bigl((2+\alpha)',2+\alpha\bigr)$. Let ${\rm{\sf N}}^r(D)$ be the null space of 
$D=D_{{}^\Vert}$ or $D_{{}^\bot}$ endowed with the $L^r$-norm.
Then the projection $P:L^r(\Omega,\Lambda)\to{\rm{\sf N}}^r(D)$ maps 
$L^r(\Omega,\Lambda)$ to $L^{r^*}(\Omega,\Lambda)$. Moreover, 
${\rm{\sf N}}^r(D)={\rm{\sf N}}^{r^*}(D)$ with equivalent norms and the projection
$P$ extends to a bounded operator from $L^p(\Omega,\Lambda)$ to
$L^{p^S}(\Omega,\Lambda)\cap {\rm{\sf N}}^p(D)$ for all 
$p\in J_\alpha$ where $J_\alpha$ denotes the open interval around 
$(2^*)'=\frac{2n}{n+1}$: $\bigl(((2+\alpha)^*)',\bigl(((2+\alpha)')^*\bigr)'\bigr)$.
\end{lemma}

\begin{proof}
Let $r\in\bigl((2+\alpha)',2+\alpha\bigr)$. The projection 
$P:L^r(\Omega,\Lambda)\to{\rm{\sf N}}^r(D)$ coming from the Hodge decomposition
\eqref{eq:HodgeLpVert} (or \eqref{eq:HodgeLpbot}) satisfies, thanks to
\eqref{eq:solvability2} (or \eqref{eq:solvability3}),
$$
\|Pu\|_{r^*}\le C\,\|Pu\|_r\le C'\,\|u\|_r
$$
since for $v\in{\rm{\sf N}}^r(D)$, we have that $dv=0$, $\delta v=0$ in 
$\Omega$ and $\nu\lrcorner\,v=0$ (or $\nu\wedge v=0$) on $\partial\Omega$. 
This proves that $P$ maps $L^r(\Omega,\Lambda)$ to $L^{r^*}(\Omega,\Lambda)$.

It is clear that ${\rm{\sf N}}^{r^*}(D)\hookrightarrow{\rm{\sf N}}^r(D)$ since we 
assumed that $\Omega$ was bounded. Conversely, let $v\in {\rm{\sf N}}^r(D)$. 
Then we have that $dv=0$, $\delta v=0$ in $\Omega$ and
$\nu\lrcorner\,v=0$ (or $\nu\wedge v=0$) on $\partial\Omega$, and thanks to
\eqref{eq:solvability2} (or \eqref{eq:solvability3}), $v\in L^{r^*}(\Omega,\Lambda)$
and $\|v\|_{r^*}\lesssim\|v\|_r$, which proves that
${\rm{\sf N}}^r(D)\hookrightarrow{\rm{\sf N}}^{r^*}(D)$ and therefore 
${\rm{\sf N}}^{r^*}(D)={\rm{\sf N}}^r(D)$ with equivalent norms. 

Let now $p\in J_\alpha$. 
We want to prove that $P$ maps $L^p(\Omega,\Lambda)$ to 
$L^{p^S}(\Omega,\Lambda)$. Since $P$ maps $L^r(\Omega,\Lambda)$
to $L^{r^*}(\Omega,\Lambda)$ for all $r\in\bigl((2+\alpha)',2+\alpha\bigr)$, its
adjoint maps $L^{\frac{nq}{n+q-1}}(\Omega,\Lambda)$ to $L^{q}(\Omega,\Lambda)$
for all $q\in\bigl((2+\alpha)',2+\alpha\bigr)$. We know moreover that $P$
is a projection, so that $P=P'=P^2$. Therefore, we obtain by composition that
$P$ maps $L^{\frac{nq}{n+q-1}}(\Omega,\Lambda)$ to 
$L^{\frac{nq}{n-1}}(\Omega,\Lambda)$ for all $q\in\bigl((2+\alpha)',2+\alpha\bigr)$.
If we let $p=\frac{nq}{n+q-1}$, we obtain that $p^S=\frac{nq}{n-1}$ and the
result is proved.
\end{proof}

To prove Theorem~\ref{thm:pHlip}, we need the following lemma which gives a
partial right inverse of $D_{{}^\Vert}$ (or $D_{{}^\bot}$) in $L^p(\Omega,\Lambda)$.

\begin{lemma}
\label{lem:inverseD}
Let $p\in J_\alpha$ ($J_\alpha$ was defined in Lemma~\ref{lem:nullspace}). 
Then any $u\in L^p(\Omega,\Lambda)$ can be decomposed as
\begin{equation}
\label{eq:decDT+K}
u=D_{{}^\Vert}Tu+Ku=D_{{}^\bot}Su+Lu
\end{equation}
where 
\begin{equation}
\label{eq:mapT}
T,S:L^p(\Omega,\Lambda)\to L^{p^S}(\Omega,\Lambda)\cap 
\begin{cases}
{\rm{\sf D}}^p(D_{{}^\Vert})\\
{\rm{\sf D}}^p(D_{{}^\bot})
\end{cases}
\end{equation}
and
\begin{equation}
\label{eq:mapK}
K,L:L^p(\Omega,\Lambda)\to L^{p^S}(\Omega,\Lambda)\cap 
\begin{cases}
{\rm{\sf N}}^p(D_{{}^\Vert})\\
{\rm{\sf N}}^p(D_{{}^\bot})
\end{cases}
\end{equation}
are bounded linear operators.
\end{lemma}

\begin{proof}
Let $D:=D_{{}^\Vert}$ or $D_{{}^\bot}$.
Let $r\in\bigl((2+\alpha)',2+\alpha\bigr)$. We denote by ${\rm{\sf D}}^r(D)$ the 
domain and ${\rm{\sf R}}^r(D)$ the range of $D$, both endowed with the 
$L^r$-norm. We have that
$$
{\rm{\sf D}}^r(D)=\begin{cases}
{\rm{\sf D}}^r(d)\cap{\rm{\sf D}}^r(\underline{\delta})\quad\mbox{if } D=D_{{}^\Vert}
\\
{\rm{\sf D}}^r(\delta)\cap{\rm{\sf D}}^r(\underline{d})\quad\mbox{if } D=D_{{}^\bot}
\end{cases}
$$ 
and since the Hodge decompositions 
\eqref{eq:HodgeLpVert}--\eqref{eq:HodgeLpbot} hold in $L^r(\Omega,\Lambda)$, 
the projection onto the null space
of $D$, $P:L^r(\Omega,\Lambda)\to {\rm{\sf N}}^r(D)$, is bounded and the operator 
$D:{\rm{\sf D}}^r(D)\to{\rm{\sf R}}^r(D)$ is invertible; we denote by 
$\widetilde{T}:{\rm{\sf R}}^r(D)\to{\rm{\sf D}}^r(D)$ its inverse. Let $p:=(r^*)'$: 
$p$ belongs to $J_\alpha$, and $p^S=r^*$ by \eqref{eq:r*'S}. From now on, we 
assume that $D=D_{{}^\Vert}$ (the case $D=D_{{}^\bot}$ can be treated similarly). 
We define $T:=({\rm I}-P)\,\widetilde{T}\,({\rm I}-P)$ and $K:=P$. It is
clear that $T$ maps $L^r(\Omega,\Lambda)$ to itself and that, thanks to
\eqref{eq:solvability2}
$$
\|Tu\|_{r^*}\le C\,\bigl(\|Tu\|_r+\|dTu\|_r+\|\delta Tu\|_r\bigr)\le C\,\|u\|_r,
\quad \forall\,u\in L^r(\Omega,\Lambda),
$$
which proves also, by duality ($T$ is self-adjoint in $L^2(\Omega,\Lambda)$),
\begin{equation}
\label{eq:Tr*'-->r'}
\|Tu\|_r\le C\,\|u\|_{\frac{nr}{n+r-1}},\quad \forall\,r\in\bigl((2+\alpha)',2+\alpha\bigr).
\end{equation}
It remains to
prove that these operators $T$ and $K$ satisfy \eqref{eq:decDT+K} and 
the mapping properties \eqref{eq:mapT} and \eqref{eq:mapK}. The fact that
$K=P$ satisfies \eqref{eq:mapK} is a direct consequence 
of Lemma~\ref{lem:nullspace}. Next, let 
$u\in L^r(\Omega,\Lambda)$, $(2+\alpha)'<r<2+\alpha$. Since $D_{{}^\Vert}P=0$
and $D_{{}^\Vert}\widetilde{T}v=v$ for all $v\in {\rm{\sf R}}^r(D_{{}^\Vert})$, 
we have that
$$
D_{{}^\Vert}Tu=D_{{}^\Vert}\widetilde{T}({\rm I}-P)u=({\rm I}-P)u=u-Ku,
$$
which proves \eqref{eq:decDT+K} for $u\in L^r(\Omega,\Lambda)$. 
The last step in this proof is to show that $T$ maps $L^p(\Omega,\Lambda)$ 
to $L^{p^S}(\Omega,\Lambda)\cap{\rm{\sf D}}^p(D_{{}^\Vert})$ for all 
$p\in J_\alpha$.
Let $u\in L^2(\Omega,\Lambda)\cap L^p(\Omega,\Lambda)$ and denote by
$w\in \dot W^{1,p}({\mathbb{R}}^n,\Lambda)$ the solution of 
$$
(d+\delta)u=\begin{cases}
({\rm I}-P)u\quad\mbox{in }\Omega\\
0\quad\mbox{outside }\Omega
\end{cases}
\in L^p({\mathbb{R}}^n,\Lambda).
$$
We have that 
$\|w_{|_{\Omega}}\|_{p^S}+
\|w_{|_{\partial\Omega}}\|_{L^{\frac{(n-1)p}{n-p}}(\partial\Omega,\Lambda)}
\le C\,\|u\|_p$.
Let now $v:=Tu-w_{|_{\Omega}}$: $v$ satisfies 
$$
\begin{cases}
(d+\delta)v=0\quad\mbox{in }\Omega,\\
\nu\lrcorner\,v=-\nu\lrcorner\,w\in B^{p,p}_{1-1/p}(\partial\Omega,\Lambda)
\hookrightarrow L^{\frac{(n-1)p}{n-p}}(\partial\Omega,\Lambda).
\end{cases}
$$
Let $q=\frac{(n-1)p}{n-p}$, so that $q^*=\frac{np}{n-p}=p^S$; in particular,
$q\in \bigl((2+\alpha)',2+\alpha\bigr)$.
By \eqref{eq:solvability2}, since $dv+\delta v=0$, we have that
$$
\|v\|_{q^*}\le C\,\bigl(\|v\|_q+\|\nu\lrcorner\,v\|_{L^q(\partial\Omega,\Lambda)}\bigr)
$$
and therefore, using \eqref{eq:Tr*'-->r'} and the fact that $\frac{nq}{n+q-1}=p$
$$
\|Tu\|_{p^S}\lesssim \bigl(\|Tu\|_q+\|u\|_p\bigr) \lesssim \|u\|_p,
$$
which ends the proof.
\end{proof}

\begin{corollary}
\label{cor:pHlip}
Let $\Omega\subset{\mathbb{R}}^n$ be a bounded strongly Lipschitz domain.
Then the operators $D_{{}^\Vert}$ and $D_{{}^\bot}$ admit a bounded 
holomorphic functional calculus on $L^p(\Omega,\Lambda)$ for all 
$p$ in the interval $\bigl(((2+\alpha)^*)',(2+\alpha)^*\bigr)$.
\end{corollary}

\begin{proof}
The proof follows the lines of the proof of Theorem~\ref{thm:mainThm4.1}.
Conditions \eqref{condA} of Theorem~\ref{thm:abstractThm} and \eqref{condB} 
of Theorem~\ref{thm:abstractThm2}
hold for $X_p=L^p(\Omega,\Lambda)$ and $A=B=D_{{}^\Vert}$ (or $D_{{}^\bot}$),
using the potentials $(T,K)$ (or $(S,L)$) defined in Lemma~\ref{lem:inverseD}.  
\end{proof}

\begin{proof}[Proof of Theorem~\ref{thm:pHlip}]
It is an immediate consequence of Corollary~\ref{cor:pHlip} and \eqref{5.1(ii)} of
Theorem~\ref{thm:mainThm4.1}.
\end{proof}

\section{Hodge-Laplacian and Hodge-Stokes operators}

Direct applications of the results in Section~\ref{sec:Dirac} are the following
properties of the Hodge-Laplacian $-\Delta_{{}^\Vert}={D_{{}^\Vert}}^2$ 
and the Hodge-Stokes operator $S_{{}^\Vert}$ defined as the part of 
$-\Delta_{{}^\Vert}$ in ${\rm{\sf N}}^2(\underline{\delta})$ extended as sectorial 
operators in $L^p(\Omega,\Lambda)$ and in ${\rm{\sf N}}^p(\underline{\delta})$. 

\begin{corollary}
\label{cor:HodgeLaplacian}
Suppose $\Omega$ is a very weakly Lipschitz domain in ${\mathbb{R}}^n$.
Define $-\Delta_{{}^\Vert}={D_{{}^\Vert}}^2$ in $L^2(\Omega,\Lambda)$. 
If $p_H<p<p^H$, then $-\Delta_{{}^\Vert}$ is sectorial of angle $0$ in 
$L^p(\Omega,\Lambda)$ and for all $\mu\in (0,\frac{\pi}{2})$, $-\Delta_{{}^\Vert}$
admits a bounded $S_{\mu +}^\circ$ holomorphic functional calculus in
$L^p(\Omega,\Lambda)$.
\end{corollary}

Let us mention that the first part of this corollary (sectoriality of $-\Delta_{{}^\Vert}$) 
has been proved in \cite{MM09a} in the case of a bounded strongly Lipschitz 
domain.

\begin{corollary}
\label{cor:HodgeStokes}
Suppose $\Omega$ is a very weakly Lipschitz domain in ${\mathbb{R}}^n$.
Define $S_{{}^\Vert}:={D_{{}^\Vert}}^2$ in 
${\rm{\sf R}}^2(\underline{\delta},\Omega)$. If $\max\bigl\{1,(p_H)_S\bigr\}<p<p^H$, 
then $S_{{}^\Vert}$ is sectorial of angle $0$ in ${\rm{\sf R}}^p(\underline{\delta})$ 
and for all $\mu\in (0,\frac{\pi}{2})$, $S_{{}^\Vert}$ admits a bounded 
$S_{\mu +}^\circ$ holomorphic functional calculus in 
${\rm{\sf R}}^p(\underline{\delta})$.
\end{corollary}

\section{General $L^p$ extrapolation results}
\label{sec:general}

In Section~\ref{sec:Dirac}, we used the following extrapolation results, 
but are presenting them separately, as they are general results which could 
be useful in other contexts. In them, $L^p(\Omega):=L^p(\Omega,{\mathbb{C}}^N)$, 
where $\Omega$ is an open subset of ${\mathbb{R}}^n$, and $N$ is a positive 
integer.

\begin{theorem}
\label{thm:abstractThm}
Let $q\in [1,\infty)$, $\max\{1,q_S\}\le p<q$ and $0\le\omega<\mu<\frac{\pi}{2}$.
Let $A$ be a bisectorial operator of angle $\omega$ in $L^q$ such that 
the family of the resolvents 
$\bigl\{({\rm I}\,+z A)^{-1},z\in{\mathbb{C}}\setminus S_\mu\bigr\}$
has $L^q-L^q$ off-diagonal bounds. Assume that $B$ is an unbounded 
operator in $L^q$ such that 
$\bigl\{({\rm I}\,+z A)^{-1}zB,z\in{\mathbb{C}}\setminus S_\mu\bigr\}$
has $L^q-L^q$ off-diagonal bounds.
\begin{enumerate}[(A)]
\item
\label{condA}
Assume that $X_p$ is a closed subspace of $L^p(\Omega)$ such that 
for all $u\in X_p$, there exist $w,v\in L^q(\Omega)$ with $w\in {\rm{\sf D}}^p(B)$,
$\|w\|_q,\|v\|_q\lesssim \|u\|_p$ and $u=Bw+v$. 

\noindent
Moreover, assume that
for each $t\in(0,{\rm diam}\,\Omega]$ there exists a family 
$\bigl\{Q_k^t,k\in{\mathbb{Z}}^n\bigr\}$ of open subsets of $\Omega$ 
with the property that 
\begin{align*}
&|Q_k^t|\lesssim t^n,\quad 
1\!{\rm l}_{\Omega}\le \sum_k1\!{\rm l}_{Q_k^t}\le 
N\,1\!{\rm l}_{\Omega}, \\
&\sup_j\sum_ke^{-\varepsilon\,{\rm dist}\,(Q_k^t,Q_j^t)/t}
=\sup_k\sum_je^{-\varepsilon\,{\rm dist}\,(Q_k^t,Q_j^t)/t}\le C_\varepsilon
\end{align*}
for all $\varepsilon>0$, where $C_\varepsilon$ does not depend on $t$,
and for all $u\in X_p$, there exist $w_k,v_k\in L^q(\Omega)$ 
such that $w_k\in {\rm{\sf D}}^p(B)$ for all $k$, and $w_k, v_k$ satisfy
$$
{\rm sppt}_\Omega w_k,{\rm sppt}_\Omega\,v_k\subset Q_k^t,
\quad
\|w_k\|_q,\|v_k\|_q\lesssim t^{1-n(\frac{1}{p}-\frac{1}{q})}\|1\!{\rm l}_{Q_k^t}u\|_p,
\quad
u=\sum_k \bigl(Bw_k+\textstyle{\frac{1}{t}}v_k\bigr).
$$
\end{enumerate}
Then there exists a constant $M_{p,\mu}$ such that 
$$
\bigl\|({\rm I}\,+zA)^{-1}u\bigr\|_p\le M_{p,\mu} \|u\|_p,\quad
\forall\,z\in{\mathbb{C}}\setminus S_\mu,\ \forall\, u\in X_p\cap L^q(\Omega)\ .
$$
\end{theorem}

\begin{proof}
For $z\in{\mathbb{C}}\setminus S_\mu$,
let $t=\min\{|z|,{\rm diam}\,\Omega\}\in(0,{\rm diam}\,\Omega]$,
$u\in X_p$. 

\noindent
If $t={\rm diam}\,\Omega$, then let $w$ and $v$ be as in
the first part of Assumption \eqref{condA}: $u=Bw+v$, and therefore
$({\rm I}\,+zA)^{-1}u=\frac{1}{z} ({\rm I}\,+zA)^{-1}zBw+({\rm I}\,+zA)^{-1}v$,
so that, thanks to the boundedness of $({\rm I}\,+zA)^{-1}$ and 
$({\rm I}\,+zA)^{-1}zB$ in $L^q(\Omega)$,
$$
\|({\rm I}\,+zA)^{-1}u\|_p
\lesssim ({\rm diam}\,\Omega)^{n(\frac{1}{p}-\frac{1}{q})}
\Bigl(\frac{1}{{\rm diam}\,\Omega}\,\|w\|_q+\|v\|_q \Bigr)\lesssim \|u\|_p.
$$

\noindent
If $t<{\rm diam}\,\Omega$, then let $Q_k^t,w_k,v_k$ 
as in the statement of the theorem. 
Then, using the $L^q-L^q$ off diagonal bounds for $({\rm I}\,+zA)^{-1}$ and
$({\rm I}\,+zA)^{-1}zB$ we have that for all $u\in X_p\cap L^q$
\begin{align*}
\|({\rm I}\,+zA)^{-1}u\|_p
&\le\Bigl(\sum_j\int_{Q_j^t}|({\rm I}\,+zA)^{-1}u|^p\Bigr)^{\frac{1}{p}}
\\
&\lesssim\Bigl[\sum_j\Bigl(\|({\rm I}\,+zA)^{-1}u\|_{L^q(Q_j^t)}
|Q_j^t|^{\frac{1}{p}-\frac{1}{q}}\Bigr)^p\Bigr]^{\frac{1}{p}}
\qquad \mbox{\small (by H\"older's inequality)}
\\
&\lesssim\Bigl[\sum_j\Bigl(\sum_k\|({\rm I}\,+zA)^{-1}(tBw_k+v_k)\|_{L^q(Q_j^t)}\,
t^{-1+n(\frac{1}{p}-\frac{1}{q})}\Bigr)^p\Bigr]^{\frac{1}{p}}
\\
&\hspace{4cm} \mbox{\small (since $u=\sum_k(Bw_k+\frac{1}{t}v_k)$ and
$|Q_j^t|\lesssim t^n$)}
\\
&\lesssim\Bigl[\sum_j\Bigl(\sum_k e^{-c\,{\rm dist}\,(Q_j^t,Q_k^t)/|z|}
\bigl(\textstyle{\frac{t}{|z|}}\,\|w_k\|_q+\|v_k\|_q\bigr)\,
t^{-1+n(\frac{1}{p}-\frac{1}{q})}\Bigr)^p\Bigr]^{\frac{1}{p}}
\\
&\hspace{4cm} \mbox{\small (by off-diagonals bounds)}
\\
&\lesssim C_c\Bigl[ \sum_k \Bigl(
\bigl(\textstyle{\frac{t}{|z|}}\,\|w_k\|_q+\|v_k\|_q\bigr)\,t^{-1+n(\frac{1}{p}-\frac{1}{q})}
\Bigr)^p\Bigr]^{\frac{1}{p}}
\\
&\hspace{4cm} \mbox{\small (by Schur's lemma and the fact that $\frac{t}{|z|}\le1$)}
\\
&\lesssim C_c\Bigl[\sum_k
\|1\!{\rm l}_{Q_k^t}u\|_p^p\Bigr]^{\frac{1}{p}}
\\
&\hspace{3cm} \mbox{\small (by the $L^q$ bounds for $w_k$ and $v_k$ 
and since $t\le|z|$)}
\\
&\lesssim \|u\|_p
\end{align*}
where we have used, in the last estimate, the finite overlapping property of
the cubes $Q_k^t$. 
\end{proof}

\begin{theorem}
\label{thm:abstractThm2} 
Suppose that all the hypotheses of Theorem~\ref{thm:abstractThm} hold, 
but with \eqref{condA} replaced by \eqref{condB}:
\begin{enumerate}[(A)]\setcounter{enumi}{1}
\item
\label{condB}
Assume that $X_p$ is a closed subspace of $L^p(\Omega)$ such that there is
a Calder\'on-Zygmund type decomposition: for all $\alpha>0$
and all $u\in X_p$ there exist functions $g, w_k, v_k\in L^q(\Omega)$, $t_k>0$ and
cubes $Q_k=Q(x_k,t_k)\subset{\mathbb{R}}^n$ of center $x_k$ and sidelength
$t_k$ such that 
\begin{align*}
&\|g\|_p\lesssim \|u\|_p,\quad \|g\|_\infty\le \alpha, 
\\  
&1\!{\rm l}_\Omega\le \sum_k1\!{\rm l}_{Q_k}\le N\,1\!{\rm l}_\Omega,\ 
\|1\!{\rm l}_{Q_k\cap\Omega}u\|_p\lesssim\alpha |Q_k|^{\frac{1}{p}},\ 
\sum_k|Q_k|\lesssim\frac{1}{\alpha^p}\,\|u\|_p^p,
\\
&{\rm sppt}\,w_k,{\rm sppt}\,v_k\subset Q_k\cap\Omega,\ 
w_k\in {\rm{\sf D}}_{L^p}(B),\quad
\|w_k\|_q,\|v_k\|_q\lesssim 
{t_k}^{1-n(\frac{1}{p}-\frac{1}{q})}\|1\!{\rm l}_{Q_k\cap\Omega}u\|_p,
\\[4pt]
\mbox{and}\quad&u=g+\sum_k\bigl(Bw_k+\textstyle{\frac{1}{t_k}}v_k\bigr).
\end{align*}
\end{enumerate}
If $A$ admits a bounded $S_\mu^\circ$ holomorphic functional calculus
in $L^q(\Omega)$, then $f(A)$ is bounded from $X_p\cap L^q(\Omega)$ 
to the weak $L^p$ space $L^p_w(\Omega)$ defined as follows
$$
L^p_w(\Omega):=\Bigl\{u:\Omega\to \Lambda \mbox{ measurable }; 
\|u\|_{p,w}:=
\bigl(\sup_{\alpha>0}\alpha^p\bigl|\bigl\{x\in \Omega;|u(x)|>\alpha
\bigr\}\bigr|\bigr)^{\frac{1}{p}}<\infty\Bigr\}
$$ 
i.e. for each $\theta\in(\omega,\mu)$ there exists $K_{p,\theta}$ such that 
$$
\|f(A) u\|_{p,w}\leq K_{p,\theta}\|f\|_\infty\|u\|_p
\quad\forall\,u\in X_p\cap L^q(\Omega), \ \forall\,f\in\Psi(S_\mu^\circ).
$$
\end{theorem}

\begin{proof}
The idea of the proof presented below is inspired by the techniques developed
in \cite{HMM11}. The starting point is a Calder\'on-Zygmund like decomposition
as \eqref{condB} in the statement.
 
It suffices to prove the result when $\|f\|_\infty=1$.  So assume henceforth that 
$\|f\|_\infty =1$.  

We proceed in several steps. Let $f\in \Psi(S_\mu^\circ)$.
Let $\alpha>0$, $u\in X_p$, and
write $u=g+\sum_k(Bw_k+\frac{1}{t_k}v_k)$ as in the statement of the theorem.

\medskip

\noindent
{\tt Step~1:}
{\em The part involving $g$.}

\noindent
We have that $g\in L^q(\Omega)$ with the estimate
$$
\|g\|_q\le \|g\|_\infty^{1-p/q}{\|g\|_p}^{p/q}
\lesssim\alpha^{1-p/q}{\|u\|_p}^{p/q}.
$$
Using the boundedness of $f(A)$ on $L^q(\Omega)$,
we have
$$
\alpha^{p}\bigl|\bigl\{x\in\Omega:
|f(A)g(x)|>\alpha\bigr\}\bigr|
\lesssim\,\alpha^{p}\frac{1}{\alpha^q}{\|f(A)g\|_q}^q
\lesssim\,\alpha^{p-q}{\|g\|_q}^q,
$$
which shows, using the bound just proven for $\|g\|_q$,
\begin{equation}
\label{eq:g}
\alpha^{p}\bigl|\bigl\{x\in\Omega:|f(A)g(x)|>\alpha\bigr\}\bigr|
\lesssim {\|u\|_p}^p
\end{equation}

\medskip

\noindent
{\tt Step~2:}
{\em On the subsets $2Q_k\cap\Omega=Q(x_k,2t_k)\cap\Omega$.}

\noindent
We denote by $E$ the set $\cup_k (2Q_k\cap\Omega)$.
We have the estimate
$|E|\le \sum_k|2Q_k|\lesssim\frac{1}{\alpha^p}\,\|u\|_p^p$,
so that 
\begin{equation}
\label{eq:E}
\alpha^p|E|\lesssim \|u\|_p^p.
\end{equation}

\medskip

\noindent
{\tt Step~3:}
{\em We claim that for all $m\ge1$, 
$$
\bigl\|\sum_kR_k^m
\bigl(Bw_k+\textstyle{\frac{1}{t_k}}v_k\bigr)\bigr\|_q
\lesssim\alpha\bigl|\displaystyle{\bigcup_kQ_k}\bigr|^{1/q},
$$
where $R_k:=({\rm I}\,+it_kA)^{-1}$ and for $M\ge 1$ to be chosen later,
\begin{equation}
\label{eq:good-bad}
\alpha^p\Bigl|\bigl\{x\in\Omega\setminus E: \bigl|f(A)\sum_k
\bigl({\rm I}\,-({\rm I}\,-R_k)^M\bigr)
\bigl(Bw_k+\textstyle{\tfrac{1}{t_k}}v_k\bigr)\bigr|>\alpha\bigr\}\Bigr|
\lesssim \|u\|_p^p.
\end{equation}
}

\noindent
Indeed, let $h\in L^{q'}(\Omega)$ with $\|h\|_{q'}=1$. 
We have that
\begin{align*}
&\Bigl|\int_\Omega
\bigl\langle\sum_kR_k^m\bigl(Bw_k+\tfrac{1}{t_k}v_k\bigr),
h\bigr\rangle\Bigr|
\\
\le&\Bigl|\int_\Omega\bigl\langle\sum_k\tfrac{1}{t_k}w_k,
t_kB^*R_k^*(R_k^*)^{m-1}h\bigl\rangle\Bigr| 
+\Bigl|\int_\Omega\bigl\langle\sum_k\tfrac{1}{t_k}v_k
(R_k^*)^mh\bigl\rangle\Bigr|, \quad\mbox{\small (taking the adjoints)}
\\[4pt]
\le&\sum_k\Bigl(\frac{1}{t_k}\|w_k\|_q
\|t_kB^*R_k^*(R_k^*)^{m-1}h\|_{L^{q'}(Q_k\cap\Omega)}+
\frac{1}{t_k}\|v_k\|_q
\|(R_k^*)^mh\|_{L^{q'}(Q_k\cap\Omega)}\Bigr).
\end{align*}
For each $k$, we denote by $A_{kj}$, $j\ge 1$, the annulus 
$2^jQ_k\setminus 2^{j-1}Q_k$ and by $A_{k0}=Q_k$, so that 
${\mathbb{R}}^n=\bigcup_{j\ge0}Q_{kj}$.
For each $k$, we decompose $h$ as 
$h=\sum_{j\ge0}1\!{\rm l}_{A_{kj}\cap\Omega}h$ and we obtain
\begin{align*}
&\Bigl|\int_\Omega \bigl\langle\sum_kR_k^m
\bigl(Bw_k+\frac{1}{t_k}v_k\bigr),h\bigr\rangle\Bigr|
\\
\lesssim&\sum_k\frac{1}{t_k}\bigl(\|w_k\|_q+\|v_k\|_q\bigr) \Bigl(\sum_je^{-c2^j}
\|h\|_{L^{q'}(\Omega\cap A_{kj})}\Bigr)
\\
&\qquad\mbox{\small (thanks to the off-diagonal bounds satisfied by 
by $R_k^*$, $t_kB^*R_k^*$ and compositions of them)}
\\
\lesssim&\sum_k\frac{1}{t_k}\|1\!{\rm l}_{Q_k}u\|_p 
t_k^{1-n(\frac{1}{p}-\frac{1}{q})}
\Bigl[\sum_je^{-c2^j}2^{jn/q'}t_k^{n/{q'}}
\Bigl(\fint_{2^jQ_k}|\tilde h|^{q'}\Bigr)^{1/{q'}}\Bigr]
\\
&\qquad\mbox{\small (using the bounds for $w_k$ and $v_k$ in $L^q$, denoting
by $\tilde h$ the extension by zero to ${\mathbb{R}}^n$ of $h$}
\\
&\qquad\mbox{\small and using the fact that $|2^jQ_k|=2^{jn}t_k^n$)} 
\\
\lesssim&\sum_k \alpha\, t_k^n
\inf_{x\in Q_k}\bigl({\mathcal{M}}(|\tilde h|^{q'})(x)\bigr)^{1/{q'}}
\Bigl(\sum_je^{-c2^j}2^{jn/p'}\Bigr)
\\
&\qquad\mbox{\small (since $\|1\!{\rm l}_{Q_k}\tilde{u}\|_p
\lesssim\alpha t_k^{n/p}$ and
using the maximal function ${\mathcal{M}}$ in ${\mathbb{R}}^n$}
\\
\lesssim&\alpha\sum_k\int_{Q_k}({\mathcal{M}}(|\tilde h|^{q'})\bigr)^{1/{q'}}
\qquad\mbox{\small (since 
$t_k^n\inf_{x\in Q_k}|f|(x)\lesssim\int_{Q_k}|f|$)}
\\
\lesssim& \alpha\int_{\bigcup_k Q_k}({\mathcal{M}}(|\tilde h|^{q'})\bigr)^{1/{q'}}
\qquad\mbox{\small (by the finite overlap property of the $Q_k$)}
\\
\lesssim&\alpha\bigl|\bigcup_kQ_k\bigr|^{1/q}\||\tilde h|^{q'}\|_1^{1/{q'}}
\qquad\quad\mbox{\small (thanks to the following estimate
(see, e.g., \cite[Lemma~5.16]{Du00}):}
\\
&\hspace{4.5cm}\mbox{\small
$\int_F\bigl({\mathcal{M}}|\varphi|\bigr)^{1/{q'}}
\lesssim |F|^{1/q}\|\varphi\|_1^{1/{q'}}$)}
\\
\lesssim& \alpha\bigl|\bigcup_kQ_k\bigr|^{1/q}.
\end{align*}
To prove \eqref{eq:good-bad} we now use the fact that $f(A)$ is bounded in 
$L^q$ and we obtain
\begin{align*}
&\alpha^p\Bigl|\bigl\{x\in\Omega\setminus E: \bigl|f(A)\sum_k
\bigl({\rm I}\,-({\rm I}\,-R_k)^M\bigr)
\bigl(Bw_k+\frac{1}{t_k}v_k\bigr)(x)\bigr|>\alpha\bigr\}\Bigr|
\\
\lesssim&\alpha^p\,\frac{1}{\alpha^q}\,\Bigl\|\sum_k
\bigl({\rm I}\,-({\rm I}\,-R_k)^M\bigr)
\bigl(Bw_k+\frac{1}{t_k}v_k\bigr)\Bigr\|_q^q
\\
\lesssim&\alpha^{p-q}\Bigl(\sum_{m=1}^M
\left(\begin{array}{c}\!\!\!M\\\!\!\!m\end{array}\!\!\!\right)
\Bigl\|\sum_kR_k^m\bigl(Bw_k+\frac{1}{t_k}v_k\bigr)\Bigr\|_q\Bigr)^q
\lesssim\alpha^p\bigl|\bigcup_kQ_k\bigr|\lesssim \|u\|_p^p.
\end{align*}

\medskip

\noindent
{\tt Step~4:} 
{\em Estimate of $\Bigl\|\sum_kf(A)({\rm I}\,-R_k)^M \bigl(Bw_k+\frac{1}{t_k}v_k\bigr)
\Bigr\|_{L^p_w(\Omega\setminus E)}$.}

\noindent
Let $\theta\in (\omega,\mu)$. Recall that for each $b\in L^p(\Omega)$, 
the definition of the functional calculus gives
$$
f(A)({\rm I}\,-R_k)^Mb=\frac{1}{2\pi i}\int_{\partial S_\theta^\circ}
f(z)\Bigl(1-\frac{1}{1+it_kz}\Bigr)^M(z{\rm I}\,-A)^{-1}b\,{\rm d}z.
$$
Using the change of variable 
$z=\frac{1}{t}e^{\pm i(\theta-\pi)}$ and $z=\frac{1}{t}e^{\pm i\theta}$
we obtain for $b_k:=t_kBw_k+v_k$
\begin{align}
&f(A)({\rm I}\,-R_k)^M \bigl(\tfrac{1}{t_k}b_k\bigr)
\nonumber\\
&=\frac{1}{2\pi i}\sum_{\varphi=\pm \theta, \pm(\pi-\theta)}\pm
\int_0^\infty \frac{1}{t_k}f(t^{-1}e^{i\varphi})
\Bigl(\frac{it_ke^{i\varphi}}{t+it_ke^{i\varphi}}\Bigr)^M
({\rm I}\,-te^{-i\varphi}A)^{-1}b_k\,\frac{{\rm d}t}{t}
\nonumber\\
&=\sum_{\varphi=\pm \theta, \pm(\pi-\theta)}
\pm \frac{1}{2\pi i}\int_0^{2t_k}\frac{1}{t_k}f(t^{-1}e^{i\varphi})
\Bigl(\frac{it_ke^{i\varphi}}{t+it_ke^{i\varphi}}\Bigr)^M
({\rm I}\,-te^{-i\varphi}A)^{-1}b_k\,\frac{{\rm d}t}{t}
\nonumber\\
&\ +\sum_{\varphi=\pm \theta, \pm(\pi-\theta)}
\pm \frac{1}{2\pi i}\int_{2t_k}^\infty \frac{1}{t_k}f(t^{-1}e^{i\varphi})
\Bigl(\frac{it_ke^{i\varphi}}{t+it_ke^{i\varphi}}\Bigr)^M
({\rm I}\,-te^{-i\varphi}A)^{-1}b_k\,\frac{{\rm d}t}{t}
\nonumber\\
&=\sum_{\varphi=\pm \theta, \pm(\pi-\vartheta)}
\bigl(F_{1,\varphi}^k(b_k)+F_{2,\varphi}^k(b_k)\bigr).
\label{eq:step4}
\end{align}

\medskip

\noindent
{\tt Step~4.1:}
{\em For $\varphi=\pm\theta$ or $\varphi=\pm(\pi-\theta)$, we claim that
\begin{equation}
\label{eq:step5.1}
\alpha^p\Bigl|\bigl\{x\in\Omega\setminus E : 
\bigl|\sum_kF_{1,\varphi}^k(b_k)(x)\bigr|>\alpha\bigr\}\Bigr|\lesssim \|u\|_p^p.
\end{equation}
}

\noindent
Let $h\in L^{p'}(\Omega)$ with $\|h\|_{p'}=1$. As before, for $j\ge 1$, we denote
by $A_{kj}$ the annulus $ 2^jQ_k\setminus 2^{j-1}Q_k$.
Using the representation of $F_{1,\varphi}^k(b_k)$, we have that 
\begin{align*}
&\Bigl|\int_{\Omega\setminus E}\bigl\langle\sum_kF_{1,\varphi}^k(b_k),
h\bigr\rangle\Bigr|
\\
\le&\frac{1}{2\pi}\Bigl|\sum_k\int_{\Omega}\int_0^{2t_k} 
f(t^{-1}e^{i\varphi})\Bigl(\frac{it_ke^{i\varphi}}{t+it_ke^{i\varphi}}\Bigr)^M
\Bigl\langle w_k,
tB^*({\rm I}\,-te^{-i\varphi}A^*)^{-1}(1\!{\rm l}_{\Omega\setminus E}h)
\Bigr\rangle\frac{{\rm d}t}{t^2}\Bigr|
\\
&+\frac{1}{2\pi}\Bigl|\sum_k\int_{\Omega}\int_0^{2t_k} 
\frac{t}{t_k}f(t^{-1}e^{i\varphi})\Bigl(\frac{it_ke^{i\varphi}}{t+it_ke^{i\varphi}}\Bigr)^M
\Bigl\langle v_k,
({\rm I}\,-te^{-i\varphi}A^*)^{-1}(1\!{\rm l}_{\Omega\setminus E}h)\Bigr\rangle
\frac{{\rm d}t}{t^2}\Bigr|
\\
&\hspace{4.5cm}\mbox{\small (using the definition of $b_k$ and duality)}
\\
&=\frac{1}{2\pi}\Bigl|\sum_k\int_0^{2t_k} \sum_{j\ge2}
\int_{\Omega}f(t^{-1}e^{i\varphi})\Bigl(\frac{it_ke^{i\varphi}}{t+it_ke^{i\varphi}}\Bigr)^M
\Bigl\langle w_k, tB^*({\rm I}\,-te^{-i\varphi}A^*)^{-1}
(1\!{\rm l}_{(\Omega\cap A_{kj})\setminus E)}h)\Bigr\rangle\frac{{\rm d}t}{t^2}\Bigr|
\\
&+\frac{1}{2\pi}\Bigl|\sum_k\int_0^{2t_k} \sum_{j\ge2}
\int_{\Omega}\frac{t}{t_k}f(t^{-1}e^{i\varphi})
\Bigl(\frac{it_ke^{i\varphi}}{t+it_ke^{i\varphi}}\Bigr)^M 
\Bigl\langle v_k, ({\rm I}\,-te^{-i\varphi}A^*)^{-1}
(1\!{\rm l}_{(\Omega\cap A_{kj})\setminus E)}h)
\Bigr\rangle\frac{{\rm d}t}{t^2}\Bigr|
\\
&\hspace{4.5cm}\mbox{\small (where we have decomposed 
$1\!{\rm l}_{\Omega\setminus E}h$ as
$\sum_{j\ge 2}1\!{\rm l}_{(\Omega\cap A_{kj})\setminus E}h$)}\ .
\end{align*}
We then obtain, denoting by $\tilde h$ the extension by $0$ to 
${\mathbb{R}}^n$ of $h$,
\begin{align}
&\Bigl|\int_{\Omega\setminus E}\bigl\langle\sum_kF_{1,\varphi}^k(b_k),
h\bigr\rangle\Bigr|
\nonumber\\
\lesssim&\|f\|_\infty\frac{1}{(\cos\theta)^M}
\sum_k\bigl(\|w_k\|_q+\|v_k\|_q\bigr) \int_0^{2t_k} \sum_{j\ge2}e^{-c2^{j-1}t_k/t}
\|1\!{\rm l}_{A_{kj}}\tilde h\|_{q'}\frac{{\rm d}t}{t^2}
\nonumber\\
&\qquad\mbox{\small (using the off-diagonal bounds for $({\rm I}\,+zA^*)^{-1}$
and for $zB^*({\rm I}\,+zA^*)^{-1}$,}
\nonumber\\
&\qquad\mbox{\small the estimate
$\Bigl|\frac{it_ke^{i\varphi}}{t+it_ke^{i\varphi}}\Bigr|\le \frac{1}{|\cos\varphi|}
=\frac{1}{\cos\vartheta}$, the fact that $\frac{t}{t_k}\le 2$ for $t\in [0,2t_k]$}
\nonumber\\
&\qquad\mbox{\small and the estimate 
$\frac{2^{j-1}t_k}{t}\ge \frac{2^j}{8}+\frac{2^jt_k}{4t}$ if $0<t<2t_k$)}
\nonumber\\
\lesssim& \sum_k\sum_{j\ge2}\alpha\,t_k^{n/p}
t_k^{1-n(\frac{1}{p}-\frac{1}{q})} e^{-c\frac{2^j}{8}}
\int_0^{2t_k}e^{-\frac{c}{4}\frac{2^jt_k}{t}}2^{\frac{nj}{q'}}|Q_k|^{1/{q'}}
\Bigl(\fint_{2^j Q_k}|\tilde h|^{q'}\Bigr)^{1/{q'}}\frac{{\rm d}t}{t^2}
\nonumber\\
&\qquad\mbox{\small (where we have used the bounds for $w_k$ and $v_k$ 
in $L^q$ and the fact that $\|1\!{\rm l}_{Q_k}u\|_p\lesssim\alpha\,t_k^{n/p}$)}
\nonumber\\
\lesssim&\alpha\sum_k|Q_k|
\inf_{x\in Q_k}\bigl({\mathcal{M}}(|\tilde h|^{p'})(x)\bigr)^{1/{p'}}
\Bigl(\sum_{j\ge2}2^{nj/{q'}}e^{-c2^j/8}\int_0^{2t_k}\frac{2^jt_k}{t}
e^{-\frac{c}{4}(2^jt_k/t)}\frac{{\rm d}t}{t}\Bigr)
\nonumber\\
&\qquad\mbox{\small (since $p'>q'$ and $|Q_k|^{1/n}\sim t_k$)}
\nonumber\\
\lesssim&\alpha\sum_k\int_{Q_k}
\bigl({\mathcal{M}}(|\tilde h|^{p'})\bigr)^{1/{p'}}
\Bigl(\sum_{j\ge2}2^{nj/{q'}}e^{-c2^j/8}\int_{2^{j-1}}^\infty e^{-cs/4}\,{\rm d}s\Bigr)
\lesssim \alpha\bigl|\bigcup_kQ_k\bigr|^{1/{p}}
\label{eq:step5.1-1}\\
&\qquad\mbox{\small (where we conclude as in Step~4, using the change of 
variable $s=\frac{2^jt_k}{t}$ in the integral}
\nonumber\\
&\qquad\mbox{\small with respect to $t$ and the fact that the sum over $j$
converges)}\ .
\nonumber
\end{align}
The estimate \eqref{eq:step5.1-1} shows that
$\displaystyle{\Bigl\|\sum_kF_{1,\varphi}^k(b_k)
\Bigr\|_{L^p(\Omega\setminus E,\Lambda)}
\lesssim \alpha\bigl|\bigcup_kQ_k\bigr|^{1/p}}$.
We can now prove \eqref{eq:step5.1}. We have that
$$
\alpha^p\Bigl|\bigl\{x\in\Omega\setminus E : 
\bigl|\sum_kF_{1,\varphi}^k(b_k)(x)\bigr|>\alpha\bigr\}\Bigr|
\lesssim\Bigl\|\sum_kF_{1,\varphi}^k(b_k)
\Bigr\|_{L^p(\Omega\setminus E)}^p
\lesssim\alpha^p\bigl|\bigcup_kQ_k\bigr|\lesssim\|u\|_p^p.
$$

\noindent
{\tt Step~4.2:}
{\em For $\varphi=\pm\theta$ or $\varphi=\pm(\pi-\theta)$ and 
$M>\frac{n}{q'}$, we claim that
\begin{equation}
\label{eq:step5.2}
\alpha^p\Bigl|\bigl\{x\in\Omega\setminus E : 
\bigl|\sum_kF_{2,\varphi}^k(b_k)(x)\bigr|>\alpha\bigr\}\Bigr|\lesssim \|u\|_p^p.
\end{equation}
}

\noindent
Let $h\in L^{p'}(\Omega)$ with $\|h\|_{p'}=1$.
We proceed as in the previous step and we obtain
\begin{align*}
&\Bigl|\int_{\Omega\setminus E}\bigl\langle\sum_kF_{2,\varphi}^k(b_k),h\bigr\rangle
\Bigr|
\\
\le&\frac{1}{2\pi}\Bigl|\sum_k\int_{2t_k}^\infty f(t^{-1}e^{i\varphi})\Bigl(
\frac{it_ke^{i\varphi}}{t+it_ke^{i\varphi}}\Bigr)^M
\sum_{j\ge2}\int_{\Omega}\Bigl\langle w_k, 
tB^*({\rm I}\,-te^{-i\varphi}A^*)^{-1}
(1\!{\rm l}_{(\Omega\cap A_{kj})\setminus E}h)\Bigr\rangle \frac{{\rm d}t}{t^2}\Bigr|
\\
&+\frac{1}{2\pi}\Bigl|\sum_k\int_{2t_k}^\infty \frac{t}{t_k} f(t^{-1}e^{i\varphi})
\Bigl(\frac{it_ke^{i\varphi}}{t+it_ke^{i\varphi}}\Bigr)^M 
\sum_{j\ge2}\int_{\Omega}\Bigl\langle v_k, ({\rm I}\,-te^{-i\varphi}A^*)^{-1}
(1\!{\rm l}_{(\Omega\cap A_{kj})\setminus E}h)\Bigr\rangle \frac{{\rm d}t}{t^2}\Bigr|
\\
&\qquad\mbox{\small (where we have used duality and the decomposition
$1\!{\rm l}_{\Omega\setminus E}h=
\sum_{j\ge 2}1\!{\rm l}_{(\Omega\cap A_{kj})\setminus E}h$)}
\\
\lesssim&\|f\|_\infty
\sum_k\bigl(\|w_k\|_p+\|v_k\|_p\bigr) \int_{2t_k}^\infty \Bigl(\frac{2t_k}{t}\Bigr)^{M-1} 
\sum_{j\ge2}e^{-c2^{j-1}t_k/t}\|1\!{\rm l}_{A_{kj}}\tilde h\|_{q'}\frac{{\rm d}t}{t^2}
\\
&\qquad\mbox{\small (thanks to the $L^q-L^q$ off-diagonal bounds satisfied by
$({\rm I}\,+zA^*)^{-1}$ and $zB^*({\rm I}\,+zA^*)^{-1}$}
\\
&\qquad\mbox{\small and the fact that 
$\bigl|\frac{it_ke^{i\varphi}}{t+it_ke^{i\varphi}}\bigr|\le \frac{2t_k}{t}$ if
$t\ge 2t_k$)}
\\
\lesssim& \alpha\sum_k|Q_k|
\inf_{x\in Q_k}\bigl({\mathcal{M}}(|\tilde h|^{p'})\bigr)^{1/{p'}}
\Bigl(\sum_{j\ge2}2^{nj/{q'}}t_k
\int_{2t_k}^\infty \Bigl(\frac{t_k}{t}\Bigr)^Me^{-\frac{c}{2}(2^jt_k/t)}
\frac{{\rm d}t}{t^2}\Bigr)
\\
&\mbox{\small where we have used the same arguments as for the proof 
of \eqref{eq:step5.1}.}
\end{align*}
To estimate the sum over $j\ge2$, we change the variable $s:=\frac{2^jt_k}{t}$
in the integral and we obtain
\begin{align*}
\sum_{j\ge2}2^{nj/{q'}}
\int_{2t_k}^\infty \Bigl(\frac{t_k}{t}\Bigr)^{M}e^{-\frac{c}{2}(2^jtk/t)}
\frac{{\rm d}t}{t}
&=\sum_{j\ge2}2^{nj/{q'}}\int_0^{2^{j-1}} 2^{-jM}s^{M}e^{-\frac{c}{2}s}\,
\frac{{\rm d}s}{s}
\\
&\le \Bigl(\int_0^\infty s^{M-1}e^{-\frac{c}{2}s}\,{\rm d}s\Bigr)
\Bigl(\sum_{j\ge2}2^{nj/{q'}} 2^{-jM}\Bigr)<\infty.
\end{align*}
The sum over $j$ is finite since we have chosen $M>n/{q'}$.
Therefore, we obtain as in the proof of \eqref{eq:step5.1-1}
$$
\Bigl|\int_{\Omega\setminus E}\bigl\langle\sum_kF_{2,\varphi}^k(b_k),h\bigr\rangle
\Bigr|
\lesssim \alpha\sum_k\int_{Q_k}
\bigl({\mathcal{M}}(|\tilde h|^{p'})\bigr)^{1/{p'}}
\lesssim \alpha\Bigl|\bigcup_kQ_k\Bigr|^{1/p}.
$$
This proves \eqref{eq:step5.2} the same way we proved \eqref{eq:step5.1}.

\medskip

\noindent
{\tt Step~5:}
{\em Conclusion: $f(A)$ maps $X_p\cap L^q$ to $L^p_w(\Omega)$.}

\noindent
Indeed, for all $\beta>0$ we have for $\alpha=\frac{\beta}{11}$
\begin{align*}
\Bigl\{x\in\Omega:\bigl|f(A)u(x)\bigr|&>\beta\Bigr\}
\subset
\Bigl\{x\in\Omega:\bigl|f(A)g(x)\bigr|>\alpha\Bigr\}
\cup E 
\\
&\cup
\Bigl\{x\in\Omega\setminus E:
\bigl|f(A)\sum_k\bigl({\rm I}\,-({\rm I}\,-R_k)^M\bigr)
\bigl(Bw_k+\frac{1}{t_k}v_k\bigr)(x)\bigr|>\alpha\Bigr\}
\\
&\cup \left(\bigcup_{\varphi=\pm\theta, \pm(\pi-\theta)}
\Bigl\{x\in\Omega\setminus E:
\bigl|\sum_k F_{1,\varphi}^k(b_k)(x)\bigr|>\alpha\Bigr\}\right)
\\
&\cup \left(\bigcup_{\varphi=\pm\theta, \pm(\pi-\theta)}
\Bigl\{x\in\Omega\setminus E:
\bigl|\sum_k F_{2,\varphi}^k(b_k)(x)\bigr|>\alpha\Bigr\}\right).
\end{align*}
We can estimate the size of each of the sets on the left hand side of the 
previous decomposition thanks to \eqref{eq:g}, \eqref{eq:E}, \eqref{eq:good-bad},
\eqref{eq:step5.1} and \eqref{eq:step5.2}. We prove that for all 
$u\in X_p$ and all $\beta>0$, we have that
$$
\beta^p\Bigl|\Bigl\{x\in\Omega:\bigl|f(A)u(x)\bigr|>\beta\Bigr\}\Bigr|
\lesssim \|u\|_p^p,
$$
which is exactly the claim.
\end{proof}

\appendix

\section{Deferred proofs}
\label{sec:appendix}

Recall the statement of Proposition~\ref{prop:lip-sim-smooth}:

\noindent
{\em
Let $\Omega\subset{\mathbb{R}}^n$ be a bounded strongly Lipschitz domain. 
Then there exists a bilipschitz map $\phi:{\mathbb{R}}^n\to{\mathbb{R}}^n$ 
where $\phi^{-1}(\Omega)=\Omega'$ is a smooth domain in ${\mathbb{R}}^n$
satisfying $\phi({\mathbb{R}}^n\setminus\overline{\Omega'})
={\mathbb{R}}^n\setminus\overline{\Omega}$ and $\phi(\partial\Omega')
=\phi(\partial\Omega)$.}

\begin{proof}
Let $\eta\in {\mathscr{C}}_c^\infty({\mathbb{R}}^{n-1})$ such that $\eta\ge 0$,
${\rm sppt}\,\eta\subset B_{n-1}(0,1)$ and $\int_{{\mathbb{R}}^{n-1}}\eta=1$. 
For $\varepsilon>0$, define 
$\eta_\varepsilon(x')=\varepsilon^{-(n-1)}\eta\bigl(\frac{x'}{\varepsilon}\bigr)$
for all $x'\in{\mathbb{R}}^{n-1}$.
By definition of a strongly Lipschitz domain, there is a covering of 
$\partial\Omega$ by $N$ open sets $V_j\subset{\mathbb{R}}^n$ ($j=1,\dots,N$)
with the following properties:
\begin{align*}
&\chi_j\in{\mathscr{C}}_c^\infty({\mathbb{R}}^n,[0,1]),\quad 
V_j=\bigl\{x\in{\mathbb{R}}^n; \chi_j(x)=1\bigr\}, \quad
{\rm sppt}\,\chi_j\subset U_j,
\\
&U_j=E_j\bigl(\prod_{k=1}^n[a_k,b_k]\bigr),\quad\mbox{where 
$a_k,b_k\in{\mathbb{R}}$ and $E_j$ is a Euclidian transformation,}
\\
&\rho_j(x)=E_j(x',x_n-g_j(x')),\ \forall\,x=(x',x_n)\in{\mathbb{R}}^n,\
g_j:{\mathbb{R}}^{n-1}\to{\mathbb{R}}\mbox{ Lipschitz continuous,}
\\
&\Omega\cap U_j=\rho_j({\mathbb{R}}^{n-1}\times(0,+\infty))\cap U_j.
\end{align*}
We fix now $j\in\{1,\dots,N\}$ and omit to write the subscript $j$. For the sake of
simplicity, we assume that $E_j$ is the identity on ${\mathbb{R}}^n$; 
if this is not the case, the modifications in the following proof are easy. We define
$$
\alpha:{\mathbb{R}}^n\to{\mathbb{R}}^n,\quad
\alpha(x)=\bigl(x',x_n-\chi(x)(g(x')-g_\varepsilon(x'))\bigr),\ x=(x',x_n),
$$
where $g_\varepsilon=\eta_\varepsilon*g-\varepsilon M$ for 
$\varepsilon<\frac{1}{M}$, $M:=\|\nabla g\|_\infty$. The map $\alpha$
is Lipschitz continuous by construction and we have, in particular,
for all $x'\in{\mathbb{R}}^{n-1}$,
\begin{align*}
|\eta_\varepsilon*g(x')-g(x')|=&\Bigl|\int_{{\mathbb{R}}^{n-1}}
\eta(y')\bigl(g(x'-\varepsilon y')-g(x')\bigr)\,{\rm d}y'\Bigr|
\\
\le&\int_{{\mathbb{R}}^{n-1}}\eta(y')\,M\,\varepsilon|y'|\,{\rm d}y' \le \varepsilon M,
\end{align*}
so that $\varepsilon M-\bigl(\eta_\varepsilon*g(x')-g(x')\bigr)\ge0$ for all
$x'\in{\mathbb{R}}^{n-1}$. Moreover we have the following properties:
\begin{enumerate}[(i)]
\item
It is straightforward to see that if $x=(x',x_n)\in V\cap\partial\Omega$, then
$\chi(x)=1$, $x_n=g(x')$ and therefore $\alpha(x)=(x',g_\varepsilon(x'))$, 
which defines a piece of a smooth hypersurface.
We have moreover that if $x\in{\mathbb{R}}^n\setminus U$, then 
$\chi(x)=0$ and then $\alpha(x)=x$.
\item
The map $\alpha:{\mathbb{R}}^n\to{\mathbb{R}}^n$ is invertible. 
Indeed, let $x'\in{\mathbb{R}}^{n-1}$. 
The function $h_{x'}:t\mapsto t-\chi (x',t)(g(x')-g_\varepsilon(x'))$
is smooth on ${\mathbb{R}}$ and its derivative is given by
$h_{x'}'(t)=1-\partial_n\chi(x',t)(g(x')-g_\varepsilon(x'))$. Choosing $\varepsilon>0$
small enough such that 
$$
\sup_{x'\in{\mathbb{R}}^{n-1},t\in{\mathbb{R}}}
\bigl|\partial_n\chi(x',t)\bigr|\le\frac{1}{4\varepsilon M},
$$
we have that $\frac{1}{2}\le h_{x'}'(t)\le \frac{3}{2}$, for all $t\in{\mathbb{R}}$, 
$x'\in{\mathbb{R}}^{n-1}$,
so that $h_{x'}:{\mathbb{R}}\to h_{x'}({\mathbb{R}})$ is strictly 
increasing, invertible and its
inverse is smooth (in the variable $t$). For $|t|$ large, $\chi(x',t)=0$. This 
implies that $h_{x'}(t)\xrightarrow[t\to-\infty]{}-\infty$
and $h_{x'}(t)\xrightarrow[t\to+\infty]{}+\infty$, and then
$h_{x'}({\mathbb{R}})={\mathbb{R}}$. Therefore, the map $\alpha$ is invertible, 
its inverse given by
$$
\alpha^{-1}:{\mathbb{R}}^n\to{\mathbb{R}}^n,\quad
\alpha^{-1}(y',y_n)=\bigl(y',h_{y'}^{-1}(y_n)\bigr).
$$
Moreover, since $h_{x'}$ is strictly increasing, we have that 
$\alpha({\mathbb{R}}^n\setminus\overline{\Omega})
={\mathbb{R}}^n\setminus\overline{\alpha(\Omega)}$ and 
$\alpha(\partial\Omega)=\partial\bigl(\alpha(\Omega)\bigr)$. 
\item
The map $\alpha^{-1}$ is Lipschitz continuous.
The Jacobian $n\times n$ matrix of $\alpha$ at a point $x=(x',t)$ is given by 
$$
J_\alpha(x',t)=\left(
\begin{array}{c|c}
{\rm I}_{n-1}&\nabla_{x'}\bigl(x'\mapsto h_{x'}(t)\bigr)
\\[4pt]
\hline
\\[-8pt]
0&1-\partial_n\chi(x',t)\bigl(g(x')-g_\varepsilon(x')\bigr)
\end{array}
\right)
$$
This matrix is invertible, its inverse at a point $(x',t)=\alpha^{-1}(y',y_n)$ is given by
$$
J_\alpha(x',t)^{-1}=\left(
\begin{array}{c|c}
{\rm I}_{n-1}&-\frac{\nabla_{x'}\bigl(x'\mapsto h_{x'}(t)\bigr)}{1
-\partial_n\chi(x',t)\bigl(g(x')-g_\varepsilon(x')\bigr)}
\\[4pt]
\hline
\\[-8pt]
0&\frac{1}{1-\partial_n\chi(x',t)\bigl(g(x')-g_\varepsilon(x')\bigr)}
\end{array}
\right)
=J_{\alpha^{-1}}(y',y_n)
$$
which is bounded on ${\mathbb{R}}^n$. Therefore $\alpha^{-1}$ is 
Lipschitz continuous.
\end{enumerate}
Following this construction for all $j=1,\dots,N$, we finally obtain
$$
\alpha:=\alpha_N\circ\dots\circ\alpha_1:
{\mathbb{R}}^n\to{\mathbb{R}}^n\quad\mbox{is a bilipschitz map}
$$
for which $\alpha(\Omega)=\Omega'$ is a smooth domain. Letting 
$\phi=\alpha^{-1}$ proves the claim made in Proposition~\ref{prop:lip-sim-smooth}.
\end{proof}

The following result shows a property of smooth domains. We didn't use it
in this paper, but it seems to us to be of independent interest and can justify,
a posteriori, together with Proposition~\ref{prop:lip-sim-smooth},
the classical assumption that for $\Omega$ a bounded
strongly Lipschitz domain, $x\in \partial\Omega$, $r>0$, the domain
$B(x,r)\cap \Omega$ has the same Lipschitz constant as $\Omega$
(see, e.g., \cite[\S5]{MM09a}).

\begin{lemma}
\label{lem:smooth-balls}
Let $\Omega'$ be a smooth domain in ${\mathbb{R}}^n$. For
$x_0\in\partial\Omega'$ and $r>0$, we consider $B(x_0,r)\cap\Omega'$.
Then there exists a smooth domain (of class ${\mathscr{C}}^3$)
$Q_r\subset{\mathbb{R}}^n$ such that
$$
B(x_0,r)\cap\Omega'\subset Q_r\subset B(x_0,2r)\cap\Omega'.
$$
\end{lemma}

\begin{proof}
We define $G:{\mathbb{R}}^n\to{\mathbb{R}}$ by
$$
G(x)=2r^2{\rm dist}\,(x,{\mathbb{R}}^n\setminus\Omega')^2
-\max\bigl\{0,(|x-x_0|^2-r^2)^2\bigr\},
\quad x\in{\mathbb{R}}^n.
$$
The function $G$ is of class ${\mathscr{C}}^3$. We define $Q_r:=G^{-1}(0,+\infty)$.
Then $Q_r$ is of class ${\mathscr{C}}^3$. It remains to verify that 
$B(x_0,r)\cap\Omega'\subset Q_r\subset B(x_0,2r)\cap\Omega'$.
\begin{enumerate}[(i)]
\item 
If $x\in B(x_0,r)\cap\Omega'$, then 
${\rm dist}\,(x,{\mathbb{R}}^n\setminus\Omega')^2>0$ and
$\max\bigl\{0,(|x-x_0|^2-r^2)^2\bigr\}=0$. Therefore, $G(x)>0$, and $x\in Q_r$.
\item
If $x\in{\mathbb{R}}^n\setminus\Omega'$, then 
${\rm dist}\,(x,{\mathbb{R}}^n\setminus\Omega')^2=0$ and therefore
$G(x)\ge 0$ which implies that $x\notin Q_r$.
\item
If $x\in \Omega'$ with $|x-x_0|\ge 2r$, then
$$
{\rm dist}\,(x,{\mathbb{R}}^n\setminus\Omega')^2\le |x-x_0|^2\quad
\mbox{and} \quad\max\bigl\{0,(|x-x_0|^2-r^2)^2\bigr\}=(|x-x_0|^2-r^2)^2.
$$
Therefore, 
$$
G(x)\le 2r^2|x-x_0|^2-(|x-x_0|^2-r^2)^2\le4r^2|x-x_0|^2-|x-x_0|^4\le0
$$
so that $x\notin Q_r$.
\end{enumerate}
This proves the properties of $Q_r$.
\end{proof}

\vspace{1cm}

\noindent
\begin{tabular}{lp{1.5cm}l}
Alan M$^{\rm c}$Intosh&&Sylvie Monniaux\\
Australian National University&&Aix-Marseille Universit\'e\\
Mathematical Science Institute&&CNRS, Centrale Marseille, I2M\\
Canberra, ACT 2601&&13453 Marseille\\
Australia&&France\\
{\tt email:\,alan.mcintosh@anu.edu.au}&&{\tt email:\,sylvie.monniaux@univ-amu.fr}
\end{tabular}

\end{document}